\pdfoutput=1

\documentclass{article}
\usepackage[utf8]{inputenc}
\usepackage{amsthm}
\usepackage{amsmath}
\usepackage{amsfonts}
\usepackage{amssymb}
\usepackage{tikz-cd}
\usepackage{biblatex}
\usepackage{cleveref}
\usepackage{color}
\usepackage{slashed}
\usepackage{graphicx}
\usepackage{subcaption}

\newtheorem{lemma}{Lemma}[subsection]
\Crefname{lemma}{Lemma}{Lemmas}
\newtheorem{prop}[lemma]{Proposition}
\Crefname{prop}{Proposition}{Propositions}
\newtheorem{thm}[lemma]{Theorem}
\Crefname{thm}{Theorem}{Theorems}
\newtheorem{thmA}{Theorem}
\Crefname{thmA}{Theorem}{Theorems}
\newtheorem{coro}[lemma]{Corollary}
\Crefname{coro}{Corollary}{Corollary}

\theoremstyle{definition}
\newtheorem{rk}[lemma]{Remark}
\Crefname{rk}{Remark}{Remarks}
\newtheorem*{notation}{Notation}
\newtheorem{eg}[lemma]{Example}
\Crefname{eg}{Example}{Examples}
\newtheorem{dfn}[lemma]{Definition}
\Crefname{dfn}{Definition}{Definitions}

\newcommand{\TZ}{TX(-\text{log}\,Z)}
\newcommand{\TZc}{T^*X(\text{log}\,Z)}
\newcommand{\Tsig}{T\Sigma(-\text{log}\,\partial \Sigma)}
\newcommand{\Z}{\mathbb{Z}}
\newcommand{\R}{\mathbb{R}}
\newcommand{\C}{\mathbb{C}}
\newcommand{\mi}{\setminus}
\newcommand{\M}{\mathcal{M}}
\newcommand{\J}{\mathcal{J}}
\newcommand{\N}{\mathbb{N}}
\newcommand{\NN}{\mathcal{N}}
\newcommand{\1}{^{-1}}
\newcommand{\uu}{{\mathbf{u}}}
\newcommand{\ind}{\text{ind}}

\graphicspath{{Pictures/}}

\title{Holomorphic curves in log-symplectic manifolds}
\author{Davide Alboresi}
\date{}

\begin{document}

\maketitle

\begin{abstract}
    Log-symplectic structures are Poisson structures that are determined by a symplectic form with logarithmic singularities. We construct moduli spaces of curves with values in a log-symplectic manifold. Among the applications, we classify symplectically ruled log-symplectic $4$ manifolds (both orientable and non-orientable), and obstruct the existence of contact boundary components, in analogy with well-known theorems by McDuff. Moreover, we study certain log-symplectically ruled surfaces, using tools from symplectic field theory.
\end{abstract}

\tableofcontents
\section{Introduction}
A log-symplectic structure on a manifold $X^{2n}$ is a Poisson tensor $\pi\in\Gamma(\bigwedge^2TX)$ for which $\pi^{\wedge n}\in\Gamma(\bigwedge^\text{2n}TX)$ vanishes transversely. Morally this means that a log-symplectic structure is very close to being a symplectic structure, as the definition implies that $\pi$ is invertible in the complement of the codimension-$1$ submanifold $Z=({\pi^{\wedge n}})\1(0)$. Hence $\omega=\pi\1$ is a symplectic form on $X\mi Z$. The singular behaviour in a neighbourhood of $Z$ is also very much constrained by the transverse vanishing requirement: there exist local coordinates such that $\omega=\frac{dx}{x}\wedge dy_1+dy_2\wedge dy_3+\dots + dy_{2n-2}\wedge dy_{2n-1}$.\\
Log-symplectic structures were introduced by Radko on orientable surfaces in \cite{radko2002}, and a general definition first appeared in \cite{guilmirpir}. A crucial observation of \cite{guilmirpir} is that one can view a log-symplectic structure as a symplectic form on a Lie algebroid, called the ``log-tangent bundle". That is a vector bundle, denoted by $\TZ$, with a Lie bracket on the space of its sections, and a compatible infinitesimal action on $X$ (this bundle is often also denoted by $^bTX$, and called the ``$b$-tangent bundle"). The Lie algebroid picture is very useful as it allows us to apply some techniques of symplectic geometry to log-symplectic structures: this is the case for Moser's argument and its consequences like the Darboux theorem, the existence of some cohomological constraints to the existence (\cite{ionut}, \cite{gil}), and also for more sophisticated results like Gompf's construction (\cite{gompfstipsicz}) of symplectic forms on Lefschetz fibrations (\cite{ralphgil1}).\\

In this paper we go further in this direction, extending the use of pseudoholomorphic curves to log-symplectic structures. We study what we call ``log-holomorphic maps", meaning maps from a surface with boundary $\Sigma$ to a log-symplectic manifold $X$ with singular locus $Z$, satisfying a holomorphicity condition. To be more precise, the log-tangent bundle $\Tsig$ of a surface relative to the boundary, as well as the log-tangent bundle $\TZ$ of a log-symplectic manifold, both carry complex structures. We study moduli spaces of maps $u:\Sigma\longrightarrow X$ such that there is an induced complex linear map $du: \Tsig\longrightarrow \TZ$. These moduli spaces turn out to behave well enough. As in symplectic geometry there are natural compactifications, and we give criteria for them to be smooth. In particular, we construct moduli spaces of spheres, and of Riemann surfaces with non-empty boundary of arbitrary genus. In order to prove compactness and smoothness we translate everything back to the symplectic world, and borrow from the general theory (especially Symplectic Field Theory).\\

As applications we deduce several results on the (symplectic) topology of log-symplectic manifolds, especially in dimension $4$, in which case the theory is more powerful. To start with, we use an argument of McDuff (\cite{mcduff}) to prove a list of non-existence results for log-symplectic manifolds (possibly with boundary). The key point is that McDuff's argument relies on a maximum principle for holomorphic curves near a contact hypersurface, which also holds in a neighbourhood of the singular locus in a log-symplectic manifold. In dimension $4$ the results can be phrased as follows.
\begin{thmA}\label{A}
	Let $(X,\,Z,\,\omega)$ be a closed $4$-dimensional log-symplectic manifold. If one component of $Z$ contains a symplectic sphere, then all components are diffeomorphic to $S^1\times S^2$. If $(X,\,Z,\,\omega)$ has non-empty boundary of contact type. Then none of the components of $Z$ contains a symplectic sphere. Moreover, the boundary cannot be contactomorphic to the standard sphere.
\end{thmA}
A finer analysis of the moduli spaces of spheres allows to prove a classification theorem up to diffeomorphism for a certain class of log-symplectic $4$-manifolds. To be more precise, we prove that if the singular locus of a log-symplectic manifold contains a symplectic sphere, then the moduli space of holomorphic spheres can be compactified adding nodal curves a $(-1)$-spheres as irreducible components. In particular the moduli space is compact whenever the log-symplectic manifold is \textit{minimal}, i.e. its symplectic locus is not a blow-up of another symplectic manifold. In that case we say that the manifold is a \textit{ruled surface}. This is the log-symplectic analogue of a famous result of McDuff (\cite{mcduffruled}). 
\begin{thmA}\label{B}
	Let $(X,\,Z,\,\omega)$ be a minimal log-symplectic $4$-manifold. Assume that $Z$ contains a copy of $S^1\times S^2$. Then $X$ supports an $S^2$-fibration over a (not necessarily orientable) closed surface $B$, with symplectic fibers, and such that $Z$ is a union of fibers. Moreover, the log-symplectic form depends only on the diffeomorphism type, up to deformations.
\end{thmA}
Moreover, the non-minimal case can be characterized in terms of Lefschetz fibrations.
It might also be worth pointing out that in the Theorems A and B above, as well as in the general construction of the moduli spaces of curves, $X$ is not assumed to be orientable.\\

The results above only use maps of spheres. As an application of our study of surfaces with boundary, we can prove a result analogous to \Cref{B}, constructing a ``log-symplectic fibration", instead of a symplectic one. We state here a simplified version of the result (for a precise statement we refer to \Cref{loglogruled}).

\begin{thmA}\label{C}
	Let $(X,\,Z,\,\omega)$ be log-symplectic, such that $Z\cong \underset{i}{\sqcup} S^1\times \Sigma_{g_i}$. Assume there exists a holomorphic sphere $S\subset X$ with $S\pitchfork Z$ and $S\cdot S=0$. Then
	\begin{itemize}
		\item $g_i=g_j$ for all $i,\,j$
		\item there is a continuous fibration $f:X\longrightarrow \Sigma_{g_i}$ with fiber $S^2$, such that $Z$ is a union of sections.
	\end{itemize} 
\end{thmA}
In fact, one can show that $f$ is smooth on $X\mi Z$, and the fibers of $f|_{X\mi Z}$ are symplectic (in this sense the fibration is ``log-symplectic"). There is also a ``non-orientable" version of this statement, producing a fibration with fiber $\R P^2$, out of a single holomorphic $\R P^2$ (holomorphic in the logarithmic sense). \Cref{C} is proven using the intersection theory of punctured holomorphic spheres (\cite{siefringwendl}, \cite{wendlpunctured}).

\paragraph{Outline of the paper.}
In section \ref{1} we quickly review the general theory of holomorphic curves. This can be skipped by the readers that are already familiar with the subject. In section \ref{2} we introduce log-symplectic manifolds, and discuss the relevant notion of holomorphicity in the log setup. In section \ref{3} we define the spaces of holomorphic curves in log-symplectic manifolds. In section \ref{4} we study compactness and smoothness of the moduli space of closed curves, and explain a few application, including \Cref{A} and \Cref{B}. In section \ref{5} we study curves with boundary and their moduli spaces, and prove \Cref{C}. In the appendix we recall some facts on the intersection theory of punctured holomorphic curves, and carry out some computations needed in section \ref{5}.

\paragraph{Acknowledgements.}
The author is thankful to Gil Cavalcanti and Chris Wendl for useful conversations, and to Marius Crainic for a question that started this project. This research was supported by the VIDI grant 639.032.221 from NWO, the Netherlands Organisation for Scientific Research.

\section{Holomorphic curves in symplectic manifolds}\label{1}

We review here some important aspects of the theory of holomorphic curves in symplectic manifolds. We will recall the relevant setup, the construction of the moduli spaces of curves, and their compactness and smoothness properties. We distinguish two main subtopics: closed curves in compact manifolds, and possibly punctured curves in open manifolds. For the former we refer to the original paper of Gromov \cite{gromov85} and the treatise \cite{mcsa}. For the latter we use the language of Symplectic Field Theory, as introduced in \cite{introsft} and \cite{compactness}; we refer to the lecture notes by Chris Wendl \cite{wendl} for an account of the relevant results.\\
We mention here once and for all that all the complex structures mentioned in the paper are \textit{not} assumed to be integrable, unless explicitely mentioned. In other words, by an (almost) complex structure on a vector bundle we always mean a fiberwise complex structure, whereas a complex structure on the tangent bundle induced by a holomorphic atlas is called integrable complex structure.

\subsection{Closed holomorphic curves}

\paragraph{Definitions and first properties.}
Let $(M,\,\omega,\,J)$ be a symplectic manifold with a compatible almost compex structure $J$. Given a closed Riemann surface $(\Sigma,\,j)$, one can consider maps 
\[
u:\Sigma\longrightarrow M
\]
such that
\[du\circ j= Jdu\] 
These maps are called \textit{holomorphic curves}. One can associate a non-negative real number to any map $u:\Sigma\longrightarrow M$, called the \textit{energy} of the curve, defined as
\[
E(u):=\int_\Sigma |du|^2 d\text{vol}_\Sigma
\]
where the norm of $du$ is defined using the Riemannian metrics on $\Sigma$ and $M$ constructed using the complex structures and respectively the symplectic forms $d\text{vol}_\Sigma$ and $\omega$. A key fact in symplectic topology is the \textit{energy identity}, stating that if $u$ is holomorphic, one has
\begin{equation}\label{energyidentity}
E(u)=\int_\Sigma u^*\omega
\end{equation}
It implies that the energy of a holomorphic curve only depends on its homology class $u_*([\Sigma])\in H_2(M;\,\Z)$ (in particular it does not depend on the choice of tame almost complex structure). As a consequence a non-constant holomorphic curve is homologically non-trivial.

\paragraph{Moduli spaces.}

One can collect holomorphic curves of the same genus in moduli spaces, i.e. one can give a geometric structure to the set of holomorphic maps from a surface with a fixed genus (up to reparametrization). More precisely, we say that two holomorphic curves $u:(\Sigma,\,j)\longrightarrow (M,\,J)$, $u':(\Sigma',\,j')\longrightarrow (M,\,J)$ are equivalent if there is a biholomorphism $\phi:(\Sigma,\,j)\longrightarrow(\Sigma',\,j')$ such that $u'\circ\phi=u$. The set of self equivalences of a holomorphic curve is called \text{automorphism group}. 
We consider
\[
\M_g(J):=\{(j,\,u): u:\Sigma_g\longrightarrow M \, \text{smooth such that } du\circ j= Jdu\}/\sim 
\]
with the topology induced by the (Fréchet) topology on the set of all smooth maps. To give a smooth structure to this space, one views the operator $\overline{\partial}_J$, defined as $\overline{\partial}_Ju=\dfrac{1}{2}(du+Jdu\circ j)$, as a section of an infinite dimensional vector bundle. This section is Fredholm. If it is transverse (i.e. its vertical component at solutions is surjective) then the zero set is smooth, as a consequence of the implicit function theorem, and finite dimensional. 
\begin{dfn}
A holomorphic curve $u$ is \textit{(Fredholm) regular} if the linearization at $u$ of the operator $\overline{\partial}_J$ is surjective.
\end{dfn}
Let $A\in H_2(M;\,\Z)$ be a singular homology class, let $u:\Sigma\longrightarrow M$. Define
\[
\M_g(A,J):=\{(j,\,u)\in \M_g(J)\,:\,[u]=A \}
\]
Given $u\in \M_g(A,J)$, we define the \textit{index of $u$} as $\text{ind}u=(n-3)(2-2g)+2c_1(A)$. This integer is also called the \textit{virtual dimension} of $\M_g(A,J)$, denoted by $\text{vdim}\M_g(A,J)$, and coincides with the Fredholm index at $u$ of the operator $\overline{\partial}_J$. By the implicit function theorem, the set $\M_g(A,J)^\text{reg}\subset\M_g(A,J)$ of Fredholm regular curves with trivial automorphism group is a smooth manifold of dimension $\text{vdim}\M_g(A,\,J)$.\\
In general one can prove regularity without abstract perturbations only for \textit{simple} curves, i.e., curves that do not factor as $u=v\circ c$, for a (non-trivial) branched cover $c:\Sigma\longrightarrow \Sigma'$. For closed simply covered curves one has the following result. Denote with $\J(\omega)$ a set of $\omega$-compatible almost complex structures on $M$. Denote with $\M^*_g(A,J)$ the set of genus $g$ simple $J$-holomorphic curves in the homology class $A$. 
\begin{thm}\label{transversalityclosed}
There exists a comeager\footnote{A countable intersection of open dense sets.} subset of complex structures $\J^\text{reg}\subset \J(\omega)$ such that for all $J\in\J^\text{reg}$ all simple $J$-curves are regular. In particular for all $J\in\J^\text{reg}$ the space $\M^*_g(A,J)$ is a smooth manifold of dimension $\text{vdim}\M_g(A,J)=(n-3)(2-2g)+2c_1(A)$
\end{thm}
There is also a version of this with \textit{marked points}. Define
\[
\M_{g,k}(A,J):=\{(j,\,u,\,(z_1,\dots,z_k)): z_j\in \Sigma_g\,:\,[u]=A \}/\sim
\]
the virtual dimension of which is $\text{vdim}\M_{g,k}(A,J)=\M_g(A,J)+2k$, and the transversality theorem applies as well.\\
On a compact manifold the moduli space of curves can be compactified adding limiting objects called \textit{nodal holomorphic curves}. Nodal curves are reducible holomorphic curves. They can be realized as maps of disconnected Riemann surfaces, with possibly singular components, and connected image. The appeareance of irreducible spherical components is called \textit{bubbling}. For a precise definition see \cite{mcsa}. Gromov's compactness theorem (\cite{gromov85}) states the following.
\begin{thm}[Gromov compactness]\label{compactnessclosed}
Let $M$ be a compact manifold, and let $J_k$ be a sequence of compatible almost complex structures. Let $u_k$ be a sequence of $J_k$-holomorphic curves satisfying a uniform energy bound $E(u_k)\leq C$. Then there exists a subsequence converging to a nodal curve $u_\infty$. If all the curves have the same genus, then also does the limit. If $[u_k]= A$, then $[u_\infty]=A$.
\end{thm}
In particular if the homology class $A$ cannot be split as a sum $A=A_1+\dots+A_N$, with $A_i$ representable by a holomorphic map then $\M_g(A,\,J)$ is compact.

\subsection{Punctured holomorphic curves}
For the study of holomorphic curves in log-symplectic manifolds we will use the language of symplectic field theory (SFT). Here we introduce the main notions from the theory, mainly with the aim of proving that one can use SFT in log-symplectic manifolds.
\paragraph{Stable Hamiltonian structures and cylindrical ends.}\label{paragraphstablehamiltonian}
The target manifolds in SFT are manifolds with \textit{cylindrical ends}, or completions of manifolds with stable Hamiltonian boundary.
\begin{dfn}\label{stablehamiltonian}
A \textit{stable Hamiltonian structure} on a $(2n-1)$-dimensional manifold $Z$ is a pair $(\alpha,\,\beta)$ consisting of a $1$-form $\alpha$ and a $2$-form $\beta$ such that
\begin{itemize}
    \item $d\beta=0$
    \item $\alpha\wedge\beta^{n-1}$ is a volume form 
    \item $\ker\beta\subset \ker d\alpha$
\end{itemize}
The \textit{Reeb vector field} of a stable Hamiltonian structure is the unique vector field $R$ such that $\iota_R\beta=0$ nad $\alpha(R)=1$.
\end{dfn}
\begin{eg}\label{examplecosymplectic}
If $\alpha$ and $\beta$ are both closed, with $\alpha\wedge\beta^{n-1}\neq 0$, then $(\alpha,\,\beta)$ is stable Hamiltonian. In this case $d(s\alpha)+\beta=ds\wedge\alpha+\beta$ is a symplectic form on $\R\times Z$. These stable Hamiltonian structures are called \textit{cosymplectic structures}. Another example is given by contact forms and their differential.
\end{eg}
\begin{dfn}\label{stablehamiltonianboundary}
	A symplectic manifold $(W,\,\omega)$ has \textit{stable Hamiltonian boundary} $\partial W=Z^+\sqcup Z^-$ if there exists a vector field $V$ (defined in a neighbourhood of the boundary), which points outwards at $Z^+$ and inwards at $Z^-$, such that $((\iota_V\omega)_{\vert T\partial W},\,\omega_{\vert T\partial W})$ is a stable Hamiltonian structure on $\partial W$, inducing the boundary orientation on $Z^+$m and the opposite orientation on $Z^-$.
\end{dfn}
Let us assume for simplicity that $Z^-=\emptyset$.
A collar neighbourhood of a stable Hamiltonian boundary is symplectomorphic to $(-\varepsilon,\,0]\times \partial W$ with symplectic form $\omega=d(s\alpha)+\beta$ (Moser argument). The symplectomorphism is obtained by realizing the vector field $V$ as $V=\dfrac{\partial}{\partial s}$. This naturally endowes $\widehat{W}:=W\cup [0,\,\infty)\times \partial W$ with a smooth structure. Moreover, the form $\omega$ extends smoothly as a closed form $\hat{\omega}$ on $\widehat{W}$, by the formula $\hat{\omega}_{\vert [0,\,\infty)\times \partial W}:=d(s\alpha)+\beta$. Assuming that $\hat{\omega}$ is symplectic (which is the case for example when $d\alpha=0$), one calls the manifold $(\widehat{W},\,\hat{\omega})$ the \textit{completion} of $(W,\,\omega)$

\paragraph{Cylindrical complex structures and holomorphic maps.}

\begin{dfn}\label{defcylindricalcomplex}
A complex structure $J$ on $\widehat{W}$ is \textit{cylindrical} if it is $\hat{\omega}$-tame, and on the cylindrical end it satisfies
\begin{itemize}
    \item $J$ is $s$-invariant.
    \item $J\partial_s=R$
    \item $J(\ker \alpha)=\ker \alpha$
    \item $J$ is $\beta$-compatible on $\ker\alpha$
\end{itemize}
\end{dfn}
Let $\dot{\Sigma}$ be a punctured Riemann surface. Let $u:\dot{\Sigma}\longrightarrow \widehat{W}$ be a holomorphic map. 
\begin{dfn}\label{sftenergy}
The \textit{energy} of a punctured holomorphic curve is the quantity $E(u):=\underset{\varphi}{\text{sup}}\int_\Sigma \hat{\omega}_\varphi$. The supremum is taken over all $\varphi:[0,\,+\infty)\longrightarrow [0,\,\varepsilon)$ with $\varphi'>0$.
\end{dfn}
Finite energy curves enjoy special geometric properties, under the following non-degeneracy assumption on the stable Hamiltonian boundary.
\begin{dfn}\label{defmorsebott}
A stable Hamiltonian structure $(\alpha,\,\beta)$ on $Z$ is \textit{Morse-Bott} if for all $T>0$ 
\begin{itemize}
    \item the set $Z_T\subset Z$ of points belonging to a $T$-periodic Reeb orbit is a closed submanifold
    \item $\text{rank } d\alpha_{\vert M_T}$ is locally constant
    \item $T_pZ_T=\ker (d_p\phi^T-1)$ for all $p\in Z_T$, where $\phi^T$ is the time-$T$ map of the Reeb flow.
\end{itemize}
\end{dfn}
\begin{eg}\label{cosymplecticmorsebott}
Take a closed disc $(D^2,\,\omega_{\text{std}})$ and a symplectic manifold $(M,\,\eta)$. The product $(D^2\times M,\,\omega_{\text{std}}+\eta)$ is a symplectic manifold whose boundary is $S^1\times M$. The boundary inherits a stable Hamiltonian structure of cosymplectic type, namely $(dt,\,\eta)$ ($t$ is the coordinate on the circle $\R/\Z$). Such structure is Morse-Bott. Indeed, the Reeb vector field is $\partial_t$, thus the Reeb orbits have integer period, and are iterations of the standard cover of the circle, with image $S^1\times {p}$. In the notation of \Cref{defmorsebott} $Z_N=S^1\times M$ for all $N\in \N$, hence it is a closed manifold. Also the second condition in \Cref{defmorsebott} is satisfied, as $\alpha=dt$ is closed. Also the Reeb flow is the identity map, so the third condition is trivially satisfied.
\end{eg}
Finite energy curves in manifolds with Morse-Bott boundary are asymptotic to Reeb orbits, meaning that for each puncture $p$ there are coordinates $(r,\,\theta)$ centered at $p$, and a Reeb orbit
$\gamma$, such that $u(re^{i\theta})\rightarrow \gamma(e^{i\theta})$ for $r\rightarrow 0$ (see \cite{bourgeoisthesis}). If $\dot{\Sigma}=\Sigma\mi \Gamma$, we say that a puncture $z\in\Gamma$ is \textit{positive} is it is asymptotic to a Reeb orbit in $Z^+$, \textit{negative} otherwise. We split the set of punctures accordingly as $\Gamma=\Gamma^+\cup\Gamma^-$.

\paragraph{Asymptotic operators.}
Let $(Z,\,\alpha,\,\beta)$ be stable Hamiltonian, with $\xi:=\ker\alpha$, and let $\gamma:S^1\longrightarrow Z$ be a Reeb orbit of period $T$. Let $J$ be a $\beta$-compatible complex structure on $\xi$.
\begin{dfn}\label{defasymptoticoperator}
	The \textit{asymptotic operator} $A_\gamma:\Gamma(\gamma^*\xi)\longrightarrow\Gamma(\gamma^*\xi)$ is defined as
	\begin{equation}\label{formulaasymptoticoperator}
	A_\gamma\eta=-J(\nabla_t-T\nabla_\eta R)
	\end{equation}
	for some symmetric connection $\nabla$ on $TZ$. We say that $A_\gamma$ is \textit{non-degenerate} if $\ker(A_\gamma)=0$. The Reeb orbit $\gamma$ is \textit{non-degenerate} if $A_\gamma$ is.
\end{dfn}
\begin{rk}
	Given a unitary trivialization $\tau:\gamma^*\xi\cong S^1\times \R ^{2n}$, asymptotic operators are exactly the ones of the form
	\begin{equation}
		A^\tau=-J_0\frac{d}{dt}-S(t)
	\end{equation}
	for a loop $S(t)$ of symmetric matrices.
\end{rk}
Asymptotic operators play a crucial role in the index theory and in the intersection theory of punctured holomorphic curves. They can be seen as limits of the linearized Cauchy-Riemann operator at the punctures.\\

Any asymptotic operator has discrete real spectrum, accumulating at $+\infty$ and $-\infty$ (\cite{wendl}).\\
If $\dim\xi=2$, each eigenvector $e_\lambda$ of a trivialized asymptotic operator (relative to an eigenvalue $\lambda$) has a well-defined winding number $\text{wind}(e_\lambda)$. One can show (\cite{Hofer1995}) that the winding number only depends on the eigenvalue, hence $\text{wind}(\lambda):=\text{wind}(e_\lambda)$ is well a defined integer. Thus, given a trivialization $\tau$, one can define the \textit{extremal winding numbers} as
\begin{equation}\label{equationextremalwinding}
\begin{aligned}
\alpha_+^\tau(A)&:=\text{min}\{\text{wind}(\lambda):\,\lambda\in\sigma(A^\tau)\cap (0,\,+\infty)\}\\
\alpha_-^\tau(A)&:=\text{max}\{\text{wind}(\lambda):\,\lambda\in\sigma(A^\tau)\cap (-\infty,\,0)\}
\end{aligned}
\end{equation} 
One defines the \textit{parity} of $A$ as
\begin{equation}\label{equationparity}
p(A):=\alpha_+^\tau(A)-\alpha_-^\tau(A)
\end{equation}
The number $p(A)$ does not depend on the trivialization $\tau$, and if $A$ is non-degenerate, is either $0$ or $1$. 
\begin{dfn}
A non-degenerate Reeb orbit is \textit{even/odd} if $p(A_\gamma)$ is $0/1$.
\end{dfn}
One can use the extremal winding numbers to define the \textit{Conley-Zehnder index} of an orbit.
\begin{dfn}\label{thmconleyzehnderwinding}
\[
\mu_{CZ}^\tau(\gamma):=\alpha_+^\tau(A_\gamma)-\alpha_-^\tau(A_\gamma)=2\alpha_-^\tau(A_\gamma)+p(A_\gamma)=2\alpha_+^\tau(A_\gamma)-p(A_\gamma)
\]
\end{dfn}
The Conley-Zehnder index can be defined for non-degenerate asymptotic operators (and hence for non-degenerate Reeb orbits) in any dimension (see \cite{salamonnotesfloer}, \cite{wendl}, \cite{Hofer1995}).

\paragraph{Moduli spaces.}\label{subsectionmodulisft}
We want to state here the dimension formula for the moduli space of punctured holomorphic curves -- this wil depend on the so-called \textit{relative Chern number} and on the Conley-Zehnder index (for possibly degenerate orbits). Consider a punctured surface $\dot{\Sigma}=\Sigma\mi\Gamma$, and let $\mathcal{R}=(\mathcal{R}_1,\,\dots, \mathcal{R}_k)$ be a $k$-tuple of connected sets of Reeb orbits. We say that a set of asymptotic constraints $\mathfrak{c}$ is the datum of
\begin{itemize}
	\item a partition $\Gamma=\Gamma^\pm_C\sqcup\Gamma^\pm_U$ into (positive or negative) constrained and unconstrained punctures;
	\item for each $z\in \Gamma_C$, a Reeb orbit $\gamma_z$;
	\item for each $z\in \Gamma_U$, a connected manifold $\mathcal{S}_z\subset \partial W$ of points belonging to a family $\mathcal{S}_z$ of Reeb orbits.
\end{itemize}
We say that a punctured holomorphic curve $u:\dot{\Sigma}\longrightarrow \widehat{W}$ satisfies the constraints $\mathfrak{c}$ if each puncture $z\in\Gamma_C$ is asymptotic to $\gamma_z$, and each puncture $z\in\Gamma_U$ is asymptotic to a Reeb orbit belonging to $\mathcal{R}_z$.\\

A curve with $k$ punctures, subject to a constraint $\mathfrak{c}$, determines a relative homology class $A\in H_2(W,\underset{z\in\Gamma_C}{\bigcup}\gamma_z\cup\underset{z\in\Gamma_U}{\bigcup}\mathcal{S}_z)$. Define
\begin{equation*}
\M_g(A,J)^\mathfrak{c}
\end{equation*}
as the set of equivalence classes of punctured genus $g$ $J$-holomorphic curves, subject to the constraint $\mathfrak{c}$, in the homology class $A$. As in the closed case, one has a generic transversality result, and a dimension formula for the moduli space.\\ 

In order to give the dimension formula, take a holomorphic map $u$. Fix a trivialization $\tau$ of $u^*TW$ in a neighbourhood of the punctures, and define $c_1^\tau(u):=c_1^\tau(u^*TW)$, the first Chern number relative to the trivialization $\tau$, as follows. For line bundles it is defined as the number of zeroes of a transverse section extending a section constant around the punctures, with respect to $\tau$. The relative Chern number is extended additively to higher rank bundles, and the result only depends on the relative homology class of $u$.\\

The Conley-Zehnder index can be defined as follows. Given a puncture $z$, and $\varepsilon>0$ suitably small, define
\begin{flalign}
\delta_z=
\left\{\begin{aligned}
& \varepsilon  &\text{if }z\in\Gamma_C\\
-& \varepsilon &\text{if }z\in\Gamma_U
\end{aligned}\right.
\end{flalign}
Define the \textit{total Conley-Zehnder index} as 
\begin{equation}\label{deftotalconleyzehnder}
\mu^\tau(\mathfrak{c}):=\sum_{z\in\Gamma^+}^{}\mu^\tau_{CZ}(\gamma_z+\delta_z)-\sum_{z\in\Gamma^-}\mu^\tau_{CZ}(\gamma_z-\delta_z)
\end{equation}
The Conley-Zehnder index of a non-degenerate asymptotic operator is stable under small perturbation, thus the above number is well-defined.\\

Define 
\begin{equation}\label{indexformulasft}
\text{vdim}\M_g(A,J)^\mathfrak{c}:=(n-3)\chi(\dot{\Sigma})+2c_1^\tau(A)+\mu(\mathfrak{c})
\end{equation} 
(if $u\in \M_g(A,J,S)$, this is also called the \textit{index of $u$}, and denoted by $\text{ind}(u)$).
\begin{thm}[\cite{wendl}, \cite{wendlpunctured}]\label{generictransversalitysft}
	Let $\widehat{W}$ be the completion of a manifold with Morse-Bott stable Hamiltonian boundary. Let $V$ be an open subset. Fix a complex structure $J_\text{fix}$ on $\widehat{W}\setminus V$, and consider the subset $\M^*_g(A,J;\,V)^\mathfrak{c}\subset \M_g(A,J;\,V)^\mathfrak{c}$ of simple curves with a point mapped to $V$. Let $\J(\omega,J_\text{fix})$ be the set of compatible almost complex structures that coincide with $J_\text{fix}$ on $\widehat{W}\setminus V$. Then there exists a comeager subset $\J(\omega,J_\text{fix})^\text{reg}\subset \J(\omega,J_\text{fix})$ such that for all $J\in \J(\omega,J_\text{fix})^\text{reg}$ the space $\M^*_g(A,J;\,V)^\mathfrak{c}$ is smooth of dimension $\text{vdim}\M_g(A,J;\,V)^\mathfrak{c}$.
\end{thm}

\paragraph{Compactness.}
As in the closed case, when considering sequences of holomorphic curves bubbling of spheres might occur; moreover, another phenomenon, called \textit{breaking}, might occur. The objects resulting from bubbling and breaking are called \textit{holomorphic buildings}.
\begin{thm}[\cite{compactness}]\label{thmsftcompactness}
A sequence $u_k$ of punctured holomorphic curves with Morse-Bott asymptotics and uniformly bounded energy admits a subsequence converging to a holomorphic building.
\end{thm}

We will not explain the definition of buildings in full generality, but restrict to a simpler setting which is enough for our purposes. Consider the symplectic manifold $\R\times Z$, and consider spheres with two punctures (i.e. cylinders), with one puncture asymptotic to $\{-\infty\}\times \gamma_-$ (negative puncture), and the other asymptotic to $\{+\infty\}\times \gamma_+$ (positive puncture).\\ 
Let $\text{pr}_Z:\R\times Z\longrightarrow Z$ be the projection onto $Z$. \textit{Broken cylinders} with values in $\R\times Z$ can be described by the following data: a finite set $u=(u_1,\,\dots,\,u_N)$ of cylinders with values in $\R\times Z$, such that: $u_1(-\infty,\,t)=\{-\infty\}\times \gamma_-$, $u_N(+\infty,\,t)=\{+\infty\}\times \gamma_+$, and $\text{pr}_Z\circ u_j(+\infty,\,t)=\text{pr}_Z\circ u_{j+1}(-\infty,\,t)$. The integer indexing the cylinders is called the \textit{level}.\\ 
For a general $\widehat{W}$, the behaviour is similar, except that $u_1$ (the ``main level") takes values in $\widehat{W}$, while $u_j$ takes values in $\R\times Z$ for all $j>1$.\\ 
A general building is, roughly, a set of curves as above, that additionally allows for the curves $u_j$ to be nodal curves.\\
When one considers curves with a bigger number of punctures and positive genus, the notation becomes more elaborate, but the behaviour of sequences is similar. See \cite{compactness}, \cite{wendl} for details.

\section{Log-symplectic manifolds}\label{2}

In this section we define log-symplectic manifolds and list some well-known properties. The properties are formulated in a way which is useful for the purpose of studying holomorphic curves.

\paragraph{Definition and first properties.}
\begin{dfn}\label{deflogsympl}
A \textit{log-symplectic manifold} is a manifold $X$ together with a Poisson bivector $\pi$, such that
\begin{itemize}
    \item $\pi^n \pitchfork 0$
    \item ${(\pi^n)}\1(0)\neq \emptyset$
\end{itemize}
\end{dfn}
\begin{rk}
A Poisson bivector satisfying the first but not the second condition in \Cref{deflogsympl} is the inverse of a symplectic form. In the literature people mostly omit the second condition in the definition of a log-symplectic structure. In \cite{gil} and \cite{ralphgil1} the authors call a Poisson bivector satisfying our \Cref{deflogsympl} a \textit{bona fide log-symplectic structure}. Since we will mostly consider log-symplectic manifolds which are not symplectic, we preferred to avoid the use of additional terminology, and put this requirement in the definition.
\end{rk}
\begin{rk}
Log-symplectic manifolds are necessarily even dimensional. We will always denote $\text{dim}X=2n$.
\end{rk}
A log-symplectic manifold contains a distinguished submanifold, namely the locus where the Poisson structure does not have maximal rank.
\begin{dfn}\label{defsingularlocus}
The \textit{singular locus} of a log-symplectic structure $\pi$ on a manifold $X$ is the codimension-$1$ submanifold $Z:={(\pi^n)}\1(0)$.
\end{dfn}
We will always denote the singular locus of a log-symplectic manifold $X$ as $Z_X$, or just $Z$.
\begin{rk}\label{remarkcoorientation}
Since the singular locus $Z$ is the zero set of a generic section of the bundle $\bigwedge^\text{top}TX$, it is coorientable if $X$ is orientable. For the same reason if $X$ is orientable, then $Z$ is trivial in homology. 
\end{rk}

\paragraph{Local forms.}

As a consequence of the Weinstein splitting theorem, for each point in $Z$ there exists a coordinate neighbourhood $V$ with coordinate functions $(x,\,y_1,\,\dots,\,y_{2n-1})$ such that
\begin{itemize}
    \item $V\cap Z=\{x=0\}$
    \item $\pi=x\partial_x\wedge\partial_{y_1}+\partial_{y_2}\wedge\partial_{y_3}+\dots+\partial_{y_{2n-2}}\wedge\partial_{y_{2n-1}}$
\end{itemize}
One could also look at the inverse of the Poisson structure, obtaining a non-smooth symplectic form 
\begin{equation}\label{localnormal}
\dfrac{dx}{x}\wedge dy_1+dy_2\wedge d{y_3}+\dots+d{y_{2n-2}}\wedge d{y_{2n-1}}
\end{equation}
(the existence of such local form is referred to as \textit{Darboux theorem}, and was originally proven in \cite{guilmirpir}).
This can actually be viewed as a smooth form in the \textit{logarithmic tangent bundle}.
\begin{dfn}
Let $(X,\,Z)$ be a manifold with a codimension-$1$ submanifold. The \textit{logarithmic tangent bundle} $\TZ$ is the rank-$2n$ vector bundle whose sheaf of section is the sheaf $\mathfrak{X}_Z(X)$ of vector fields on $X$ tangent to $Z$. We call its dual \textit{logarithmic cotangent bundle}, and we denote it with $\TZc$.
\end{dfn}
This bundle exists and is unique up to isomorphism by the Serre-Swan theorem. The inclusion $\mathfrak{X}_Z \hookrightarrow \mathfrak{X}$ induces a map $\rho:\TZ\longrightarrow TX$, called the \textit{anchor map}, which is an isomorphism almost everywhere (namely, on $X\mi Z$). $\TZ$ is in fact a Lie algebroid.
\begin{dfn}
A \textit{logarithmic differential form} is a section of $\bigwedge^\bullet(\TZc)$. We say that a logarithmic differential form $\omega$ is \textit{closed} if $d\omega=0$, where $d$ denotes the Lie algebroid exterior differential.  
\end{dfn}
\begin{rk}
$d\omega=0$ is equivalent to $d((\rho|_{X\mi Z}\1)^*\omega)=0$, as an ordinary differential form.
\end{rk}
The local form \Cref{localnormal} implies that one can view a log-symplectic structure as a logarithmic $2$-form. This implies in particular that the vector bundle $\TZ$ admits a complex structure. To be more precise, it admits a contractible space of compatible complex structures, and a contractible space of tame complex structures (\cite{mcsa}).\\
The correspondence between log-symplectic structures and logarithmic symplectic forms is in fact one to one (\cite{guilmirpir}). Based on this, from now on we will treat a log-symplectic manifold as a triple $(X,\,Z,\,\omega)$ consisting of:
\begin{itemize}
    \item a smooth manifold $X$
    \item a codimension-$1$ submanifold $Z$
    \item a logarithmic symplectic form $\omega$ on $\TZ$
\end{itemize}

\paragraph{Maps of log-symplectic manifolds.}

\begin{dfn}[\cite{ralphgil1}]\label{logmaps}
Let $(X,\,Z_X),\,(Y,\,Z_Y)$ be pairs of manifold with a codimension-$1$ submanifold. A \textit{(smooth) map of pairs} $f:(Y,\,Z_Y)\longrightarrow (X,\,Z_X)$ is a smooth map $f:Y\longrightarrow X$  such that $f\pitchfork Z_X$, and $f\1(Z_X)=Z_Y$.
\end{dfn}
\begin{eg}\label{Zvuoto}
$f:(Y,\,\emptyset)\longrightarrow (X,\,Z)$ is a smooth map of pairs if and only if $f$ is an ordinary smooth map $f:Y\longrightarrow X\mi Z$.
\end{eg}
The differential of a smooth map of pairs lifts uniquely to the logarithmic tangent bundles, meaning that there is a commutative diagram
\begin{equation}\label{diagramlift}
\begin{tikzcd}
TY(-\log Z_Y) \arrow{r}{\overline{df}} \arrow{d} & TX(-\log Z_X) \arrow{d}\\
TY\arrow{r}{df} & TX
\end{tikzcd}
\end{equation}
The restricted bundle $\TZ|_Z$ carries a ``canonical transverse section" $\xi_X=\xi$, constructed as follows. Choose a local coordinate system $\phi=(x,\,y_1,\,\dots,\,y_{2n-1}))$ on an open set $U$ with $x$ a defining function for $Z\cap U$ and $y_i$ a coordinate chart for $Z$. Define $\xi|_U:=\phi_*(x\partial_x)$. Given another chart $\phi'=(x',\,y'_1,\,\dots,\,y'_{2n-1}))$ as above, it is easy to compute that $\phi_*(x\partial_x)|_Z=\phi'_*(x'\partial_{x'})|_Z$, which ensures that $\xi$ is well defined on $Z$. This canonical section is preserved under maps of pairs:
\begin{prop}\label{bnormal}
Let $f:(Y,\,Z_Y)\longrightarrow (X,\,Z_X)$ be a smooth map of pairs. Then $f_*(\xi_{Z_Y})=\xi_{Z_X}$.
\end{prop}
\begin{proof}
Consider a local defining function $x$ for $Z_X$. By transversality, its pullback $y=f^*x$ via $f$ is a local defining function for $Z_Y$. By definition we can compute $\xi_{Z_X}$ and $\xi_{Z_Y}$ by looking at $x\dfrac{\partial}{\partial x}$ and $y\dfrac{\partial}{\partial y}$. $f_*(y\dfrac{\partial}{\partial y})=x\dfrac{\partial}{\partial x}+xV$, where $V$ is tangent to $Z_X$. Restricting to $Z_X$ gives the desired equality.
\end{proof} 
The following lemma due to Cavalcanti and Klaasse \cite{ralphgil1} is useful in order to define embeddings of pairs, and in particular log-symplectic submanifolds.
\begin{lemma}[\cite{ralphgil1}]
Let $f:(Y,\,Z_Y)\longrightarrow (X,\,Z_X)$ be a smooth map of pairs, such that $f\1(Z_X)=Z_Y$. The anchor map $\rho: TY(-\text{log}\,Z_Y)\longrightarrow TY$ induces an isomorphism $\rho:\ker \overline{df}\longrightarrow \ker df$ at all points, where $\overline{df}$ is the lift of $df$, in the sense of (\ref{diagramlift}). 
\end{lemma}
This means that an embedding induces an inclusion at the level of the logarithmic tangent bundles. Hence the following definition makes sense.
\begin{dfn}\label{submanifold}
A \textit{log-symplectic submanifold} of $(X,\,Z_X,\,\omega)$ is a pair of manifolds $(Y,\,Z_Y)$ with a smooth embedding $f:(Y,\,Z_Y)\longrightarrow (X,\,Z_X)$ such that $\omega$ restricts to a symplectic form (i.e. nondegenerately) to $f_*(TY(-\text{log}\,Z_Y))$.
\end{dfn}
The same meaning can be given to the expression \textit{complex submanifold} of a manifold with a logarithmic complex structure. The following fact is obvious:
\begin{prop}
Let $J$ be an $\omega$-compatible complex structure on $\TZ$. Any complex submanifold is symplectic. Moreover, for each log-symplectic submanifold $(Y,\,Z_Y)$ there exists a compatible complex structure on $\TZ$ for which $(Y,\,Z_Y)$ is a complex submanifold.
\end{prop}

\paragraph{A normal form around the boundary, and the stable Hamiltonian geometry of the singular locus.}
Fix a Riemannian metric on $X$ and consider the function ``distance from $Z$", denoted with $\lambda$. We can write the log-symplectic form in a neighbourhood of $Z$ as
\[
\omega=\dfrac{d\lambda}{\lambda}\wedge a + b
\]
for $a$ and $b$ respectively a $1$- and a $2$-form (in the ordinary sense) on a neighbourhood of $Z$. It follows easily from $\omega$ being symplectic that $\alpha:=a|_{TZ}$ and $\beta:=b|_{TZ}$ determine a \textit{cosymplectic structure} on $Z$ (see \Cref{examplecosymplectic}).
The cosymplectic structure on $Z$ completely determines the log-symplectic structure in a neighbourhood of $Z$.
\begin{prop}[\cite{guilmirpir}, \cite{ionut}]\label{normalformaroundz}
Let $(X,\,Z,\,\omega)$ be log-symplectic, with $Z$ compact. There exists a cosymplectic structure $(\alpha,\,\beta)$ on $Z$, and a metric $g$ on $NZ$ such that a tubular neighbourhood of $Z$ is isomorphic to the normal bundle $NZ$ with the log-symplectic form
\begin{equation}\label{normalform0}
\dfrac{d\lambda}{\lambda}\wedge\alpha+\beta
\end{equation}
and $\lambda$ is the $g$-distance from the zero section of $NZ$.
\end{prop}
As a consequence, there exists a vector field $V$ on $X$ -- defined by $\alpha(V)=\beta(V)=0$, $\dfrac{d\lambda}{\lambda}=1$ -- such that $L_V\omega=0$, and $V$ is transverse to the level sets of $\lambda$, for values of $\lambda$ small enough. Let $U$ be a tubular neighbourhood of $Z$, with smooth boundary, so that $V\pitchfork \partial U$. The condition $L_V\omega=0$ implies that $W:=X\mi U$ is a symplectic manifold with stable Hamiltonian boundary $-\partial U$. The induced stable Hamiltonian structure is pulled back from $(-\alpha,\,\beta)$ as above, via the normal bundle projection. Via the change of coordinates $s:=-\text{log}\lambda$, one can write $\omega=ds\wedge (-\alpha)+\beta$, as in \Cref{examplecosymplectic}. Hence one has the following proposition, which is crucial to set up a theory of holomorphic curves.

\begin{prop}\label{stablehamiltonianform}
The complement $X\mi Z$ of the singular locus is symplectomorphic to the completion of the manifold $W:=X\mi U$, with stable Hamiltonian boundary of cosymplectic type. $\partial W$ is a double cover of $Z$ (disconnected if and only if $X$ is orientable) and the cosymplectic structure on the boundary coincides (up to isomorphism) with the pullback via the covering map of the cosymplectic structure on $Z$.
\end{prop}

\paragraph{An alternative normal form around the singular locus.}

One could write a more concrete normal model around $Z$, just by writing explicitely what the normal bundle looks like. To this end, note that every real line bundle over $Z$ can be written as a fiber product $\R\times_{\Z_2}\tilde{Z}$, where $c:\tilde{Z}\longrightarrow Z$ is a double cover of $Z$. More precisely, a double cover $\tilde{Z}$ of $Z$ is acted on by an involution $\sigma$, and we define an action of $\Z_2=\{\pm 1\}$ on $\R\times\tilde{Z}$ as
\begin{equation}\label{action}
    -1\cdot (x,\,z):=(-x,\,\sigma(z))
\end{equation}
and define 
\begin{equation}\label{fiberproduct}
    \R\times_{\Z_2}\tilde{Z}:=(\R\times\tilde{Z})/\Z_2
\end{equation}
Given a real line bundle over $Z$, one recovers the manifold $\tilde{Z}$ as the set of vectors of length $1$ with respect to some fiber metric, and the involution $\sigma$ is the multiplication by $-1$.\\
For a cosymplectic $Z$, $\tilde{Z}$ inherits a cosymplectic structure (via pullback along $c$) which is invariant under $\sigma$.
Hence the log-symplectic structure $\omega:=\dfrac{dx}{x}\wedge c^*\alpha+c^*\beta$ on $\R\times\tilde{Z}$ is invariant under the action (\ref{action}). Thus $\omega$ descends to the quotient. The resulting log-symplectic structure on $\R\times_{\Z_2}\tilde{Z}$ has the form
\begin{equation}\label{normalform1}
    \dfrac{dx}{x}\wedge\alpha+\beta
\end{equation}
and is symplectomorphic to the normal form (\ref{normalform0}).
This implies:
\begin{coro}\label{simplenormalform}
Let $(X,\,Z,\,\omega)$ be compact log-symplectic. There exists a double cover $c:\tilde{Z}\longrightarrow Z$ such that a neighbourhood of the singular locus can be written (up to symplectomorphism) as 
\[
((-\varepsilon,\,\varepsilon)\times_{\Z_2}\tilde{Z},\,\dfrac{dx}{x}\wedge c^*\alpha+c^*\beta)
\]
where the fiber product is taken using the deck transformation.
\end{coro}
\begin{notation}
We will most of the time omit the pullback sign from the formula for the symplectic form.
\end{notation}
Note that if $\tilde{Z}$ is the trivial double cover, then the action of $\Z_2$ just exchanges corresponding connected components, and the resulting quotient is the trivial bundle. Since this happens if and only if $Z$ is coorientable, and $NZ\cong \bigwedge^{\text{top}}TX\vert_Z$ (see \Cref{remarkcoorientation}), we obtain the following obvious specialization.
\begin{coro}\label{simplenormalformorientable}
If $X$ is orientable, a neighbourhood of the singular locus can be written (up to symplectomorphism) as 
\[
((-\varepsilon,\,\varepsilon)\times Z,\,\dfrac{dx}{x}\wedge \alpha+\beta)
\]
\end{coro}

\paragraph{Simple codimension-$1$ foliations.}

The cosymplectic structure on the singular locus induces a codimension-$1$ symplectic foliation. Assume for the moment that $Z$ is compact and connected. There is a map
\begin{equation}\label{fibration}
p:Z\longrightarrow \R/\alpha(H_1(Z,\,\Z))
\end{equation}
defined as $p(z):=\int_{z_0}^{z}\alpha$. Whenever $\alpha$ is a multiple of a rational form, then $\R/\alpha(H_1(Z,\,\Z))=\R/c\Z$ for some positive number $c$. In this case, one can prove that the map is a fibration, such that the connected components of the fibers are leaves. Moreover, one can this map to realize the leaves of the foliation $\ker\alpha$ as the fibers of a fibration.
\begin{prop}[\cite{tischler}]
Let $(Z,\,\alpha,\,\beta)$ be a cosymplectic manifold such that $\alpha$ is a real multiple of a rational form. Then $Z$ is a symplectic mapping torus; more precisely, there exists a closed symplectic manifold $(F,\,\beta_F)$, a symplectomorphism $\phi:(F,\,\beta_F)\longrightarrow (F,\,\beta_F)$, and a constant $c$, such that $(\R\times F/\Z,\,dt,\,\beta_F)\cong (Z,\,\alpha,\,\beta)$; here the $\Z$ action is generated by $1\cdot (t,\,z)=(t+c,\,\phi(z))$. 
\end{prop}
If $Z$ is disconnected, each of its components $Z_i$ inherits a fibration over $\R/c_i\Z$. The number $c_i$ is a Poisson geometric invariant, called the \textit{period of the modular vector field} (\cite{radko2002}, \cite{guilmirpir}). Log-symplectic manifolds whose singular locus is a fibration are called \textit{proper} (\cite{gil}). It is simple to observe that for a given log-symplectic form $\omega$ there exists a nearby log-symplectic form $\omega'$ such that the resulting log-symplectic manifold is proper.

\section{Holomorphic curves in log-symplectic manifolds}\label{3}

In this section we study some general aspects of holomorphic curves in log-symplectic manifolds. We start by introducing the class of almost complex structures that we will consider (namely the cylindrical ones). Then we will introduce holomorphic curves, and a notion of energy for those; we prove that log-holomorphic curves coincide with the punctured curves in SFT, and that also our notion of energy coincides the notion of energy that is used in standard Gromov-Witten theory and SFT. Finally, we prove a maximum principle which is going to be crucial for our construction of the moduli spaces. The construction of the moduli spaces is carried out in the later sections, separately for closed and open curves.

\subsection{Cylindrical complex structures on log-symplectic manifolds}

Let us recall the notation from section $3$: $U$ denotes a tubular neighbourhood of $Z$, with a fixed isomorphisms to the unit disc bundle of the normal bundle $NZ$ of $Z$. Let $\lambda$ denote the length of the fiber coordinate (with respect to some fiberwise Riemannian metric). We know from \Cref{normalformaroundz} and \Cref{stablehamiltonianform} that there exist closed forms $\alpha$, $\beta$ on $Z$, respectively a $1$- and a $2$-form, such that the log-symplectic form is 
\begin{equation}\label{normal}
\omega=d\log\lambda\wedge\alpha+\beta
\end{equation}
Viewing $X\mi Z$ as $\widehat{W}:=\widehat{X\mi U}$, it is natural to consider $\omega$-compatible almost complex structures on $T(X\mi Z)$ which are cylindrical, as in \Cref{defcylindricalcomplex}. For reasons that will become apparent, we actually use a slightly modified definition of cylindrical complex structure. Recall that each connected component $Z_i$ of $Z$ has a well defined period $c_i\in (0,\,+\infty)$. 
\begin{dfn}\label{defcylindricalcomplexlog}
Let $(X,\,Z,\,\omega)$ be log-symplectic. A complex structure $J$ on $X\mi Z$ is \textit{cylindrical} if it is $\hat{\omega}$-tame, and on each component $U_i$ of the cylindrical end $U$ it satisfies
\begin{itemize}
    \item $J$ is $s$-invariant.
    \item $J\partial_s=c_iR$
    \item $J(\ker \alpha)=\ker \alpha$
    \item $J$ is $\beta$-compatible on $\ker\alpha$
\end{itemize}
\end{dfn}
\begin{notation}
For simplicity, we will denote with $cR$ the vector field on $Z$ defined as $c_iR$ on each component $Z_i$. The above definition hence requires $J\partial_s=cR$.
\end{notation}
We need also to impose a further compatibility condition with $Z$. Recall that $U$ is isomorphic to $\R\times_{\Z_2}\tilde{Z}$, where $\Z_2$ acts by multiplication by $-1$ on $\R$, and by an involution $\sigma$ on $\tilde{Z}$ (\Cref{simplenormalform}). Hence a cylindrical complex structure is an $\R$ invariant complex structure on $(\R\times \tilde{Z}\mi \{0\}\times \tilde{Z})/\Z_2$. In particular it induces an almost complex structure on the codimension-$1$ foliation $\tilde{F}$ on $\tilde{Z}$.
\begin{dfn}
A cylindrical complex structure is \textit{adapted} to the log-symplectic structure if its restriction to $\tilde{F}$ is $\sigma$-invariant.
\end{dfn}
\begin{eg}
If $X$ is orientable, then $\tilde{Z}$ is a union of two disjoint copies of $Z$. The definition is asking for the complex structures on two corresponding copies of $F$ to be equal.
\end{eg}
The definition is designed so that the complex structure extends over $Z$ as a complex structure on $\TZ$. 
\begin{prop}\label{propcylindricalcomplex}
An adapted cylindrical almost complex structure induces a smooth complex structure on $\TZ$.
\end{prop}
Moreover, an adapted complex structure $J$ on $\TZ$ induces a complex structure on $\ker\alpha\subset TZ$, compatibly with the leafwise symplectic form $\beta$ if $J$ is $\omega$-compatible. Indeed, consider the canonical section $\xi\in\Gamma(\TZ|_Z)$. Consider the logarithmic $1$-form $\rho^*\alpha\in\Gamma(\TZc)$: $\xi\in\ker(\rho^*\alpha)$, and $\rho:\ker(\rho^*\alpha)|_Z\longrightarrow \ker(\alpha)|_Z\subset TZ$ is well-defined and surjective. Moreover, $J(\xi)\pitchfork \ker(\rho^*\alpha)$. Further, the symplectic orthogonal to the subspace generated by $\xi$ and $J(\xi)$ is contained in $\ker(\rho^*\alpha)$, and projects isomorphically to $\ker(\alpha)$. Finally, $\text{span}(\xi,\,J(\xi))^{\perp\omega}$ is closed under $J$, hence $\rho$ induces an almost complex structure on $\ker(\alpha)$. This complex structure is the same that is induced by the quotient map $\tilde{Z}\longrightarrow \tilde{Z}/\Z_2=Z$ (this follows from the normal form).\\
Moreover, since $\rho (J(\lambda\partial_\lambda))=c\tilde{R}$ the Reeb vector field for all $\lambda\neq 0$ by definition, then $\rho(J(\xi))=cR$ on $Z$.
\begin{notation}
We will normally use adapted compatible complex structures, and leave the word ``adapted" implicit. We will use the expression ``compatible complex structure" for both the almost complex structure on $X\mi Z$ and its extension to $\TZ$.
\end{notation}

\subsection{Riemann surfaces and log-holomorphic structures}

In this section we will specialize the definition of cylindrical complex structure from the previous section to $2$-dimensional log-symplectic manifolds $(\Sigma,\,\partial\Sigma)$, for which the singular locus coincides with the boundary. We explain how these ``logarithmic Riemann surfaces" correspond to punctured Riemann surfaces, and that their automorphisms correspond to automorphisms of punctured surfaces.

\paragraph{Logarithmic holomorphic structures.}
Let $\Sigma$ be a compact surface, possibly with boundary. Consider the log-tangent bundle of the pair $(\Sigma,\,\partial \Sigma)$, denoted by $\Tsig$.
\begin{dfn}\label{riemann}
A \textit{(logarithmic) holomorphic structure} on $(\Sigma,\,\partial \Sigma)$ is a complex structure $j$ on $\Tsig$, such that there exist a positive defining function $r$ for $\partial\Sigma$, and a global parametrization $\theta$ of the boundary, such that $j(r\partial_r)=\partial_\theta$ in a collar neighbourhood of $\partial\Sigma$ 
\end{dfn}
\begin{rk}
Normally a holomorphic structure on a surface with boundary is defined as a conformal structure on its double, commuting with the canonical involution. This notion is different from\Cref{riemann}, as a logarithmic holomorphic structure does not extend to a holomorphic structure on the double surface (rather, it extends to a degenerate one). Nonetheless, as we will only deal with logarithmic holomorphic structures, we will often omit the word ``logarithmic". 
\end{rk}
\begin{rk}
A logarithmic holomorphic structure on $(\Sigma,\,\partial\Sigma)$ is cylindrical, and adapted to a log-symplectic structure which is of the form $\frac{dr}{r}\wedge d\theta$ in a neighbourhood of $\partial\Sigma$.
\end{rk}

\paragraph{Punctured surfaces.}
\Cref{riemann} implies that a collar neighbourhood of a connected component of the boundary is biholomorphic to $[0,\,\epsilon)\times S^1$, $j(r\partial_r)=\partial_\theta$. There is a map
\begin{equation}\label{pigreco}
\pi: [0,\,\epsilon)\times S^1\longrightarrow D_\varepsilon
\end{equation}
where $D_\varepsilon$ is the disc of radius epsilon, defined as 
\[
\pi(r,\,\theta):=re^{i\theta}
\]
A simple computation shows that $\pi$ is a biholomorphism when restricted to the interior. Globally, using (\ref{pigreco}) one can construct a closed Riemann surface $\hat{\Sigma}$, with a map
\[
\pi:\Sigma \longrightarrow \hat{\Sigma}.
\]
The closed surface $\hat{\Sigma}$ comes with a finite set of distinguished points $p_1,\,\dots,\,p_k$, namely $\pi(\partial\Sigma)$. The restriction $\pi:\Sigma\setminus\partial\Sigma\longrightarrow \hat{\Sigma}\setminus\{p_1,\,\dots,\,p_k\}:=\dot{\Sigma}$ is a biholomorphism.\\
Conversely, given a closed Riemann surface $\hat{\Sigma}$ with a finite set of punctures $p_1,\,\dots,\,p_k$, one can view local charts centered at the punctures as a choice of cylindrical ends for $\dot{\Sigma}:=\hat{\Sigma}\setminus \{p_1,\,\dots,\,p_k\}$. One can compactify the punctured surface $\dot{\Sigma}$ adding circles at infinity -- the resulting surface with boundary $\Sigma$ has a logarithmic holomorphic structure. 

%\begin{notation}
%We will always denote by $\Sigma$ a compact surface with boundary, and a logarithmic holomorphic structure. We will denote by $\hat{\Sigma}$ the associated closed surface, and by $\dot{\Sigma}$ the associated punctured surface.
%\end{notation}

\paragraph{Maps of Riemann surfaces.}

Surfaces with a holomorphic structure are related by the following natural notion of map.

\begin{dfn}
A map of pairs $(\Sigma,\,\partial\Sigma)\longrightarrow (\Sigma',\,\partial\Sigma')$ between surfaces with logarithmic holomorphic structures is \textit{holomorphic} if the induced map on the log-tangent bundles is complex linear.
\end{dfn}
\begin{lemma}
\begin{itemize}
\item[(1)] given $\Sigma$, the corresponding closed surface $\hat{\Sigma}$ is the unique closed surface with the property that there is a map $\pi:\Sigma \longrightarrow \hat{\Sigma}$ which is an isomorphism when restricted: $\pi:\Sigma\setminus\partial\Sigma\longrightarrow \hat{\Sigma}\setminus\{p_1,\,\dots,\,p_k\}$. Unique means that given two surfaces with these properties there exists a unique isomorphism commuting with the projections.
\item[(2)] given a holomorphic map $\Sigma\longrightarrow\Sigma'$ there is a unique holomorphic map $\hat{\Sigma}\longrightarrow\hat{\Sigma}'$ making the following diagram commute
\begin{equation}
\begin{tikzcd}
\Sigma \arrow{r} \arrow{d}[swap]{\pi}     &\Sigma'\arrow{d}{\pi}\\
 \hat{\Sigma}\arrow{r} & \hat{\Sigma}'
\end{tikzcd}
\end{equation}
\item[(3)] given a holomorphic map $\hat{\Sigma}\longrightarrow\hat{\Sigma}'$ there is a unique holomorphic map $\Sigma\longrightarrow\Sigma'$ making the following diagram commute
\begin{equation}
\begin{tikzcd}
\Sigma \arrow{r} \arrow{d}[swap]{\pi}     &\Sigma'\arrow{d}{\pi}\\
 \hat{\Sigma}\arrow{r} & \hat{\Sigma}'
\end{tikzcd}
\end{equation}
\end{itemize}
\end{lemma}
\begin{proof}{(1)}
If there are two surfaces with a map $\pi:\Sigma \longrightarrow \hat{\Sigma}$ and $\pi':\Sigma \longrightarrow \hat{\Sigma}'$, then there is an isomorphism in the complement of a finite number of points (namely $\pi'\circ\pi^{-1}_{|\Sigma\setminus\{p_1,\,\dots,\,p_k\}}$). Bounded neighbourhoods of the punctures isomorphic to punctured discs are mapped to bounded neighbourhoods of the punctures. These maps extend as holomorphic automorphisms of the disc, so we obtain an isomorphism $\phi:\Sigma\longrightarrow\Sigma'$ such that $\pi\phi=\pi'$.
\end{proof}
\begin{proof}{(2), (3)}
Similar to the above, using also the fact that the only automorphisms of the punctured disc are rotations.
\end{proof}

\subsection{Holomorphic curves}

We are ready to introduce our notion of holomorphic curve in the log-setting. Fix a tubular neighbourhood $U$ of the singular locus $Z$ in $X$, and pick a compatible cylindrical complex structure. 
Let $(\Sigma,\,\partial\Sigma)$ be a logarithmic Riemann surface; recall that a map of pairs $u:(\Sigma,\,\partial\Sigma)\longrightarrow (X,\,Z)$ induces by definition a diagram 
\[
\begin{tikzcd}
\Tsig \arrow{r}{\overline{du}} \arrow{d}     &\TZ\arrow{d}\\
 T\Sigma\arrow{r}{du} & TX
\end{tikzcd}
\]
\begin{dfn}
A map of pairs $u:(\Sigma,\,\partial\Sigma)\longrightarrow (X,\,Z)$ is \textit{log-holomorphic} (or simply \textit{holomorphic}) if the lift of the differential $\overline{du}:\Tsig\longrightarrow \TZ$ is complex linear.
\end{dfn}

The choice of using cylindrical complex structures implies that the image of the boundary of a holomorphic curve is constrained: it needs to coincide with a Reeb orbit.
\begin{prop}\label{asympt}
Let $u:(\Sigma,\,\partial\Sigma)\longrightarrow (X,\,Z)$ be holomorphic. Then the boundary is mapped to a simple periodic orbit of the Reeb vector field. 
\end{prop}
\begin{proof}
This follows from \Cref{bnormal}, and the fact that $\rho(J(-\lambda\partial_\lambda))=\rho(J(\partial_s))=cR$.
\end{proof}
We can also easily compute the period of the Reeb orbit: write $Z=\underset{i}{\sqcup}Z_i$, with $Z_i$ connected. Let $c_i>0$ such that $\alpha(H_1(Z_i))=c_i\Z$ (the period of the component $Z_i$). Then the period of the Reeb orbit to which a component of the boundary is mapped only depends on the connected components of $Z$ that contains it, and it coincides with $c_i$. In particular the orbit has minimal period, hence it is simply covered. Notice that a Reeb orbit $\gamma$ coming as boundary condition of a holomorphic curve actually lifts to a closed loop $\tilde{\gamma}$ in $\tilde{Z}$. The cosymplectic structure on $\tilde{Z}$ is the pull-back of the one on $Z$, $(p^*\alpha,\,p^*\beta)$. Hence the period of $\tilde{\gamma}$ is also $c_i$, and is also the minimal period. 
\begin{prop}
A log-holomorphic curve with non-empty boundary is simply covered.
\end{prop}
\begin{proof}
Our remarks above imply that a log-holomorphic is simply covered in a neighbourhood of the singular locus. It is a general fact that a holomorphic curve is simply covered if and only if it is so in an open neighbourhood (\cite{mcsa})
\end{proof}

\begin{rk}\label{logequalpunctured}
Consider the punctured surface $\dot{\Sigma}$ associated to $\Sigma$. A holomorphic map $u:(\Sigma,\,\partial\Sigma)\longrightarrow (X,\,Z)$ induces (by restriction) a map $u:\dot{\Sigma}\longrightarrow X\setminus Z$ which is holomorphic in the usual sense. The converse holds for simple maps if we assume that the holomorphic curve has finite energy (see \cite{bourgeoisthesis}).
\end{rk}

\subsection{Energy}\label{sectionenergy}

We give a definition of energy of a holomorphic curve in a log-symplectic manifold, and prove that it coincides with the SFT energy of the curve in the symplectic locus. We keep the same notation as in the previous section.
\begin{dfn}[\cite{ionutobstruction}]\label{mu}
Consider a function $\mu:X\longrightarrow \R$, smooth on $X\mi Z$, that coincides with $\lambda$ in a neighbourhood of $Z$, vanishing only at $\lambda=0$, and such that $\mu-1$ is compactly supported in $U$. Define the form $\beta^\mu$ so that the following equality holds:
\[
\omega=d\log\mu\wedge\alpha+\beta^\mu 
\]
\end{dfn}
\begin{rk}
The cohomology class of $\beta^\mu$ is well defined and non-zero.
\end{rk}
\begin{dfn}[Energy]\label{energy}
Let $u:\Sigma\longrightarrow X$ be a smooth map. Define the \textit{energy} of $u$ as
\[
E(u):=\int_\Sigma u^*\beta^\mu +\int_{\partial\Sigma}u^*\alpha
\]
\end{dfn}
\begin{prop}
$E(u)$ is well defined.
\end{prop}
\begin{proof}
Choose two functions $\mu$, $\mu'$ as in \Cref{mu}. One sees immediately that
\[
\beta^\mu-\beta^{\mu'}=d(\log(\dfrac{\mu}{\mu}')\alpha)
\]
By Stokes' formula
\[
\int_\Sigma u^*\beta^\mu-\int_\Sigma u^*\beta^{\mu'}=\int_{\partial\Sigma}u^*\log(\dfrac{\mu}{\mu}')\alpha=0
\]
where the last equality holds as $\mu=\mu'$ in a neighbourhood of $Z$.
\end{proof}

\begin{rk}
Since $\beta^\mu$ is closed, the energy of $u$ depends only on $u(\partial\Sigma)$ and on the homology class of $u$ rel $u(\partial\Sigma)$. 
\end{rk}
The energy of a closed holomorphic curve $\Sigma$ mapped to the symplectic manifold $(X\mi Z,\,\omega|_{X\mi Z})$ coincides with the usual notion from symplectic geometry.
\begin{prop}\label{energyclosed}
If $\Sigma$ is closed and $u:\Sigma\longrightarrow X$ is holomorphic, with image contained in $X\mi Z$, then
\[
E(u)=\int_\Sigma u^*\omega
\]
\end{prop}
\begin{proof}
The image of $u$ is contained in the complement of some open set $V$ containing $Z$. We can choose $\mu$ to be equal to $1$ on $X\mi V$.
\end{proof}
We prove here that \Cref{sftenergy} is equivalent to \Cref{energy}. Notice that for log-symplectic manifolds it is actually unnecessary to take the supremum in \Cref{sftenergy}, since $\alpha$ is closed.
\begin{prop}\label{energyequalsenergy}
$E_{SFT}(u)=\varepsilon E(u)$. In particular, the two notions are equivalent, and coincide if we choose $\varepsilon=1$.
\end{prop}
\begin{proof}
Note first that $\omega_{|_{X\mi U}}={\omega_\varphi}_{|_{X\mi U}}$, hence we only need to check the equality in the cylindrical end $U$. To this aim, we assume without loss of generality that $u^{-1}(U)$ is a Riemann surface with smooth boundary. On $U$ we can write $\omega=d\log\lambda\wedge\alpha+\beta$. Choosing a function $\mu$ as in \Cref{mu}, one has $\omega=d\log(\mu)\wedge\alpha+\beta^\mu=\omega=d\text{log}\lambda\wedge\alpha+\beta$, hence
\[
E(u):=\int_{\partial \Sigma}u^*\alpha+\int_{u^{-1}(U)} u^*\beta^\mu
\]
\[=\int_{\partial \Sigma}u^*\alpha+\int_{u^{-1}(U)} u^*d(\log\lambda-\log\mu)\wedge u^*\alpha+\int_{u^{-1}(U)} u^*\beta\]
\[
=\int_{\partial \Sigma}u^*\alpha+\int_{u^{-1}(U)} u^*d(\log(\dfrac{\lambda}{\mu}))\wedge u^*\alpha+\int_{u^{-1}(U)} u^*\beta
\]
Noting that $\partial u^{-1}(U)=-u^{-1}(\partial \overline{U})\cup\partial\Sigma$, Stokes' theorem gives
\[
E(u)=\int_{\partial \Sigma}u^*\alpha+\int_{u^{-1}(U)} u^*\beta-\int_{u^{-1}(\partial \overline{U})} u^*\log(\dfrac{\lambda}{\mu})\wedge u^*\alpha
+\int_{\partial\Sigma} u^*\log(\dfrac{\lambda}{\mu})\wedge u^*\alpha
\]
Now, $\lambda\equiv\mu$ near $Z$, which implies $\int_{\partial\Sigma} u^*\log(\dfrac{-\lambda}{\mu})\wedge u^*\alpha=0$.
As $\mu\equiv 1$ near $\partial\overline{U}$, one has that $\int_{u^{-1}(\partial \overline{U})} u^*\log(\dfrac{\lambda}{\mu})\wedge u^*\alpha=\int_{u^{-1}(\partial \overline{U})} u^*\log\lambda\wedge u^*\alpha$. Hence

\begin{equation}\label{eu} 
E(u)=\int_{\partial \Sigma}u^*\alpha+\int_{u^{-1}(U)} u^*\beta-\int_{u^{-1}(\partial \overline{U})} u^*\log\lambda\wedge u^*\alpha
\end{equation}

Let us turn to the SFT energy. Writing out explicitly what the SFT energy is, one gets
\[
E_{SFT}(u)=\int_{u^{-1}(U)}u^*d(\varphi(-\log(\lambda)))\wedge u^*\alpha+u^*\beta 
\]
Comparing with \Cref{eu}, one sees that $ E_{SFT}(u)=\varepsilon E(u)$ if and only if

\[
-\int_{u^{-1}(\partial \overline{U})}u^*\varphi(-\log\lambda)\wedge u^*\alpha+\int_{\partial\Sigma}u^*\varphi(-\log\lambda)\wedge u^*\alpha\]
\[=
\varepsilon \Big[  \int_{\partial \Sigma}u^*\alpha+\int_{u^{-1}(\partial \overline{U})} u^*\log\lambda\wedge u^*\alpha
 \Big]
\]

To conclude, we only need to observe that $\varphi=identity$ near $\partial \overline{U}$, and $\varphi(-log\lambda)\rightarrow\varepsilon$ as $\lambda\rightarrow 0$ (i.e. on $\partial\Sigma$).
\end{proof}

\subsection{The maximum and minimum principle}

We prove in this section a maximum principle for holomorphic maps. This is crucial in the proof of all compactness theorems for families of holomorphic curves.\\
Realize $X\setminus Z$ as $\widehat{W}$, as in \Cref{stablehamiltonianform}. Assume $X$ is endowed with a cylindrical complex structure. Restrict to the cylindrical end $[0,\,+\infty)\times \partial W$, and consider the first projection $p_1: [0,\,+\infty)\times \partial W\longrightarrow [0,\,+\infty)$. Denote by $D^2_r$ the open disc of radius $r$.
\begin{prop}\label{max}
Let $u: D^2_r\longrightarrow [0,\,+\infty)\times \partial W$ be holomorphic. Then $p_1\circ u$ is harmonic.
\end{prop}
\begin{proof}
$\Delta(p_1\circ u)=-dd^c(p_1\circ u)=d(d(p_1\circ u)\circ J)=d(dp_1\circ J\circ du)=d(\alpha\circ du)=u^*d\alpha=0$.
\end{proof}

\begin{coro}\label{leaf}
Let $\Sigma$ be a closed Riemann surface. Let $u:\Sigma\longrightarrow X$ be holomorphic and such that $u(\Sigma)\cap U\neq \emptyset$. Then $\Sigma$ is mapped to a symplectic leaf in $\{ \lambda\}\times\partial W$, $\lambda\neq 0$.
\end{coro}

\begin{coro}
Let $\Sigma$ be compact with non-empty connected boundary, $u:\Sigma\longrightarrow X$ holomorphic with $u\1(Z)=\partial\Sigma$. Then $u(\Sigma)\cap (X\setminus U)\neq \emptyset$.
\end{coro}
\begin{proof}
If $X\mi U\cap u(\Sigma)=\emptyset$, then $p_1\circ u$ has a minimum, in contradiction with \Cref{max}.
\end{proof}

\section{Closed curves}\label{4}

In this section we study the moduli space of closed curves in a log-symplectic manifold, and prove several results on the topology of log-symplectic manifolds, including \Cref{A} and \Cref{B}.

\subsection{A simpler case: genus $0$, aspherical $Z$}

Let us begin with a discussion of a simpler setup: we look at holomorphic spheres in manifolds where there can be no holomorphic spheres in a neighbourhood of the singular locus $Z$.
\begin{dfn}
We say $(Z,\,\alpha,\,\beta)$ is \textit{symplectically aspherical} if $A\cdot \beta=0$ for all homotopy classes $A\in \pi_2(Z,\,z_0)$ for all $z_0\in Z$. 
\end{dfn}

Manifolds for which $\pi_2(Z_i)=0$ for all components $Z_i$ of $Z$ are symplectically aspherical. In the proper case, since $\pi_2(Z)\cong\pi_2(F)$ for a symplectic fiber $F$, the definition coincides with the usual asphericity condition for the symplectic manifold $(F,\,\beta)$.\\
If $(Z,\,\alpha,\,\beta)$ is aspherical, then all double covers $p:\tilde{Z}\longrightarrow Z$ with the pulled-back cosymplectic structure are. By the energy identity and \Cref{energyequalsenergy}, there exists no holomorphic sphere $u:S^2\longrightarrow \tilde{Z}$. The maximum and minimum principles (\Cref{max}) imply the following:
\begin{coro}
Let $(X,\,Z,\,\omega)$  be log-symplectic, with $Z$ symplectically aspherical. Then all non-constant holomorphic spheres are mapped into $X\setminus U$.
\end{coro}
This specializes, for instance, to:
\begin{coro}
Assume $X$ has dimension $4$, and $Z$ has no component diffeomorphic to $S^1\times S^2$. Then all non-constant holomorphic spheres are mapped into $X\setminus U$.
\end{coro}
When $Z$ is aspherical, then, all holomorphic spheres live in a compact subset of the symplectic locus of $X$. Moreover, they all intersect (and are, in fact, contained in) the subset of the symplectic locus where the almost complex structure is allowed to vary arbitrarily. Thus, the transversality for the moduli space can be achieved for a generic choice of complex structure (\Cref{transversalityclosed}). Further, since all the curves take value in a compact subset, Gromov's compactness theorem also applies (\Cref{compactnessclosed}).
\begin{prop}\label{regaspherical1}
Let $(X,\,Z,\,\omega)$  be log-symplectic, with $Z$ symplectically aspherical. Let $A\in H_2(X)$ be a spherical homology class. Then there exist a comeager set of cylindrical complex structures $\J_\text{reg}$ such that for all $J\in \J_\text{reg}$ the moduli space $\M^*_{0,m}(J,\,A)$ is smooth of dimension $2n-6+2c_1(A)+2m$, and it is compact modulo bubbling.
\end{prop}
In particular, if bubbling can be prevented (for example by considering classes $A$ such that each non-trivial partition $A=A_1+\dots+A_k$ has $A_i\cdot\omega\leq 0$) then $\M^*_0(A)$ is a compact smooth manifold of dimension $\text{vdim}\M(A)$. In general, bubbling can be controlled if $X\mi Z$ is low-dimensional ($\text{dim}X=4,\,6$), or is, more generally, a \textit{semipositive manifold} (\cite{mcsa}).
\begin{dfn}
A symplectic manifold $(M,\,\omega)$ is \textit{semipositive} if for all $A\in\pi_2(M)$, $3-n\leq c_1(A)<0$ implies $\omega\cdot A\leq 0$.
\end{dfn}
This is automatically satisfied in dimension $\leq 6$. The usefulness of this definition lies in the fact that in semipositive manifolds, and for generic complex structures, bubbling can only happen in codimension two. This means that fixing a homology class $A$, the set of nodal curves representing $A$ has codimension at least $2$ in the space of all curves representing $A$.\\
For convenience we will say that log-symplectic manifold $(X,\,Z,\,\omega)$ is \textit{semipositive} $(X\mi Z,\,\omega|_{X\mi Z})$ is semipositive.

\subsection{Obstructing certain log-symplectic manifolds}\label{mcduff}

A standard argument by McDuff (\cite{mcduffruled}) can be used to obstuct the existence of certain classes of log-symplectic manifolds, of which \Cref{A} mentioned in the Introduction is a special case. Let us summarize McDuff's argument here. We keep on assuming that $Z$ is aspherical, and that $X\mi Z$ is semipositive.\\
Consider $A\in H_2(X\mi Z)$ with $c_1(A)=2$, so that $\text{vdim}\M^*_{0,1}(A,J)=2n$, and pick a generic $J$. Consider the evaluation map $\text{ev}: \M^*_{0,1}(A,J)\longrightarrow X\mi Z$. Assuming for the moment that the evaluation $\text{ev}$ is a proper map, we can compute its degree. Since $Z$ is aspherical there is no holomorphic sphere in a neighbourhood of $Z$, thus $\text{deg}(ev)=0$.\\ 
However, in some explicit example one might be able to prove that there is a point $x\in X\mi Z$ such that there exists a unique Fredholm regular holomorphic sphere passing through that point, leading to a contradiction with the asphericity of $Z$.\\
For instance, one can prove that such holomorphic spheres exist in the following situations (see \cite{mcduffruled} for details and a more general discussion):
\begin{itemize}
    \item if there exist symplectic submanifolds with trivial normal bundle, symplectomorphic to $S^2\times V$, with $V$ a K\"ahler manifold with $\pi_2(V)\cdot\omega=0$. Here extend the product complex structure on $S^2\times V$ to the whole manifold, take $A=[S^2\times \{p\}]$ and apply \cite{mcsa}, Chapter 3.3, to show that the obvious holomorphic sphere if Fredholm regular.
    \item there is a boundary component of contact type, contactomorphic to $S^{2n-1}$ with the standard contact structure. Here realize the sphere as the boundary of a Darboux ball in $S^2\times\dots\times S^2$, and take $A=[S^2\times\{p\}]$
    \item if $\text{dim}X\leq 6$, a codimension-$2$ symplectic submanifold $S$ symplectomorphic to $\mathbb{C}P^{n-1}$, with $c_1(NS)\geq 0$. Blow-up a copy of $\C P^{n-2}$ until the normal bundle to $S$ becomes trivial, and take the proper transform of a $\C P^1\subset \C P^{2n-1}$ transverse to $\mathbb{C}P^{n-2}$.
\end{itemize}
If $\text{ev}$ is not proper, but $X \mi Z$ is semipositive, one can argue as follows. Pick a point $x$ in a neighbourhood of $Z$, and a point $x'\in X\mi Z$ (in the same component as $x$) such that there is an element in $\M^*_{0,1}(A,J)$ passing through $x$. Let $N:=ev(\overline{\M}_0(A,J)\mi ev(\M^*_0(A,J))$ be the set of points in the image of a nodal curve. This set has codimension at least $2$ by semipositivity, hence there is a path $\gamma$ with values in $(X\mi Z)\mi N$ joining $x$ and $x'$. There is a neighbourhood of $\gamma$ which does not intersect $N$, and the evaluation map is proper over this neighbourhood (no nodal curves are mapped to this neighbourhood). Now we can argue as in the compact case above.\\
As examples of applications, we point out the following corollaries.
\begin{coro}\label{obstruction2}
Let $(V,\,\sigma_{V})$ be a symplectic manifold such that $\pi_2(V)\cdot\sigma_{V}=0$, and such that $V_k$ supports a $\sigma_{V}$-compatible integrable complex structure. Let $\sigma_{S^2}$ be any symplectic form on $S^2$. Let $(X,\,Z,\,\omega)$ be a closed semipositive log-symplectic manifold, such that a component of $Z$ is symplectomorphic to a mapping torus over $(S^2\times V,\,\sigma_{S^2}\times \sigma_{V})$. Then none of the components of $Z$ are symplectically aspherical.
\end{coro}
Given that any symplectic submanifold $W$ with trivial normal bundle can be used to produce a singular component diffeomorphic to $S^1\times W$ (\cite{gil}), the conclusion of the \Cref{obstruction2} holds if one assumes that $X$ contains a symplectic $(S^2\times V,\,\sigma_{S^2}\times \sigma_{V})$ with trivial normal bundle.\\
A log-symplectic manifold with boundary is a manifold $X$ with boundary $\partial X$ with a codimension-$1$ submanifold $Z$ such that $\partial X\cap Z=\emptyset$, together with a symplectic form $\omega$ on $\TZ$. 
\begin{coro}
Assume that $(X,\,Z,\,\omega)$ is compact, semipositive, and has non-empty boundary of contact type, contactomorphic to the standard $(2n-1)$-dimensional sphere. Then $Z$ cannot be symplectically aspherical.
\end{coro}
The following is also an easy corollary of McDuff's argument, applied to log-symplectic manifolds with non-aspherical singular locus.
\begin{coro}
Let $(V_k,\,\sigma_{V_k})$ be symplectic manifolds such that $\pi_2(V_k)\cdot\sigma_{V_k}=0$, and such that $V_k$ supports a $\sigma_{V_k}$-compatible integrable complex structure. Let $\sigma_{S^2}$ be any symplectic form on $S^2$. Let $Z$ be a union of symplectic mapping tori over $(S^2\times V_k,\,\sigma_{S^2}\times \sigma_{V_k})$. Then there exists no log-symplectic manifold $(X,\,Z,\,\omega)$ with non-empty boundary of contact type.
\end{coro}
\begin{proof}
Let $U$ be a neighbourhood of the singular locus, with smooth boundary. The boundary of $U$ is also a mapping torus over $S^2\times \tilde{V}_k$, for some double cover $\tilde{V}_k$. We can use the $S^2$ factor to produce a holomorphic curve with Chern number $2$. By the properties of the moduli space of spheres there must then be a holomorphic curve with values in a neighbourhood of the boundary, which is a contradiction with the contact type hypothesis. More precisely, due to a maximum principle analogous to \Cref{max} any sphere with a point mapped close to the boundary must be entirely contained in a neighbourhood of the boundary. Moreover, one can ensure that there exists a sphere mapped in a neighbourhood of the boundary where the symplectic form is exact. By the energy identity this implies that the holomorphic sphere has $0$ energy, and is therefore constant, which contradicts the fact that it has Chern number $2$.
\end{proof}
Specializing the above discussion to the $4$-dimensional case, one obtains the following statements, which provide a strengthening of \Cref{A}. Recall that every $4$-manifold is semipositive.
\begin{coro}\label{obstruction0}
If a log-symplectic $4$-manifold contains a symplectic sphere with $0$ self-intersection, then the singular locus is a union of copies of $S^1\times S^2$.
\end{coro}
\begin{coro}\label{obstruction}
If one component of the singular locus of a closed $4$-dimensional log-symplectic manifold is diffeomorphic to $S^1\times S^2$, then all components are diffeomorphic to $S^1\times S^2$.
\end{coro}
\begin{coro}
In a symplectic $4$-manifold, symplectic spheres with non-negative normal Chern number, and symplectic surfaces of positive genus with trivial normal bundle, must intersect (in particular they cannot be homologous). 
\end{coro}
\begin{proof}
This is because around a symplectic surface $\Sigma$ with trivial normal bundle the symplectic structure can be modified into a log-symplectic form, with singular locus $S^1\times \Sigma$ (\cite{gil}, Theorem 5.1).
\end{proof}

\begin{coro}
Let $(X,\,Z,\,\omega)$ be compact log-symplectic $4$-manifold with non-empty boundary of contact type. Then $Z$ is a union of $S^1\times S^2$.
\end{coro}

\begin{coro}
Let $(X,\,Z,\,\omega)$ be compact log-symplectic $4$-manifold with non-empty boundary of contact type. Then $\partial X$ is not contactomorphic to the standard contact $3$-sphere.
\end{coro}

\subsection{Compactness of the moduli space}

Let us now consider an arbitrary \textbf{proper} log-symplectic manifold. In order to compactify the moduli space of curves one needs to take care not only of bubbling, but also of the non-compactness caused by the presence of the singular locus, as shown in the following simple example.
\begin{eg}
Consider $(X,\,Z)=(\R\times S^1\times \Sigma,\,\{0\}\times S^1\times \Sigma)$, with coordinates $(x,t,z)$, and complex structure $J(x\partial_x)=\partial_t$, $J|_{T\Sigma}=j_\Sigma$ on $\TZ$. The inclusion $u_{(x,t)}$ of $\Sigma$ in $X$ as $\{(x,t)\}\times\Sigma$ is a log-holomorphic map for all $x\neq 0$. Letting $x\rightarrow 0$, one finds the map $u_{(0,t)}$, the inclusion of $\Sigma$ in $X$ as $\{(0,t)\}\times\Sigma$. This is not a map of pairs, as in \Cref{logmaps}. 
\end{eg}
This example suggests that we might want to add the following set to the moduli space:
\begin{equation}\label{moduliZ}
\M_g(A,J;Z):=\Big\{[(u,j)] \,:\, 
\begin{aligned}
   & u:\Sigma_g\longrightarrow Z \text{ has values in a symplectic leaf $F$}\\
   &u \text{ is $(j,J)$-holomorphic as an $F$-valued map}
\end{aligned} \Big\}/\sim
\end{equation}
The notion of holomorphicity as an $F$-valued map is well-defined, thanks to the discussion following \Cref{propcylindricalcomplex} - a cylindrical complex structure on $\TZ$ induces a compatible leafwise complex structure on $\ker\alpha\subset TZ$.\\ 
In order to show that one can compactify the moduli space by adding (possibly nodal) holomorphic curves in $Z$, and nodal curves in $X\mi Z$, one applies a very simple idea: one just observes that a sequence of holomorphic curves has a subsequence contained in the compact symplectic manifold $X\mi U$, or has a subsequence contained in a compact neighbourhood of $Z$. In both cases Gromov's theorem gives us a convergent subsequence. In the remainder of this section we make this idea precise.
\begin{lemma}\label{compactnessZ}
$\M_g(A,J;\,\tilde{Z})$ is compact modulo nodal curves.
\end{lemma}
\begin{proof}
In the proper case, this is an immediate consequence of \Cref{compactnessclosed}, applied to a (compact) fiber $\tilde{F}$ with varying complex structure $\tilde{F}_t$. 
\end{proof}
\begin{rk}
Of course one does not need to fix the homology class $A$: it is enough to uniformly bound the energy. 
\end{rk}

\begin{thm}\label{compactnessclosedlog}
Let $(X,\,Z,\,\omega)$ be a closed log-symplectic manifold. $\M_g(A,J)\sqcup \M_g(A,J;Z)$ is compact modulo nodal curves
\end{thm}
\begin{proof}
Take a sequence $u_k$ of elements in $\M_g(A,J)\sqcup \M_g(A,J;Z)$. If there is a subsequence with values in $\M_g(A,J;Z)$, the statement is a consequence of \Cref{compactnessZ}, hence we can assume that $u_k\in \M_g(A,J)$. Let $\tilde{U}\cong (-\varepsilon,\,\varepsilon)\times \tilde{Z}$. Let $U:=\tilde{U}/\Z_2$. A sequence of holomorphic curves $u_k$ satisfies one of the following two:
\begin{itemize}
    \item it has a subsequence contained in $X\mi U$
    \item it has a subsequence contained in $\overline{U}\mi Z$
\end{itemize}
In the first case, there is a subsequence converging to a nodal curve in $X\mi U$, by the standard compactness result. In the second case, there are non-zero real numbers $\lambda_k$, and holomorphic maps $v_k:\Sigma\longrightarrow\tilde{Z}$, such that $u_k=(\lambda_k,\,v_k)$. Up to a subsequence, this converges in $\tilde{U}=(-\varepsilon,\,\varepsilon)\times \tilde{Z}$ to $u_\infty=(\lambda_\infty,\,v_\infty)$, with $\lambda_\infty$ a possibly zero real number, and $v_\infty$ a nodal holomorphic curve in $\tilde{Z}$. Denoting the projection to the quotient with $q:\tilde{U}\longrightarrow U$, we find that $q\circ u_\infty$ is the limit (in $U\subset X$) of $u_k=q\circ u_k$ (up to subsequence).
\end{proof}
The proof shows that we do not need to add the whole of $\M_g(A,J;\,Z)$ in order to compactify $\M_g(A,J)$. It is enough to take the union of $\M_g(A,J)$ the following set:
\begin{equation}\label{modulinearZ}
    \mathcal{N}(A,J;U):=\{q\circ u \,:\,u:\Sigma\longrightarrow (-\varepsilon,\,\varepsilon)\times\tilde{Z} \}/\sim
\end{equation}
It is readily seen, using \Cref{leaf}, that
\begin{equation}\label{isomodulinearZ}
    \mathcal{N}(A,J;U)=(-\varepsilon,\,\varepsilon)\times_{\Z_2}\M_g(A,J;\,\tilde{Z})
\end{equation}
where the action of $\Z_2$ on $\tilde{Z}$ is by post-composition with the involution $\sigma$ (recall that we chose a $\sigma$-equivariant $J$).
\begin{dfn}\label{defpartialcompactification}
Denote with $\mathcal{N}_g(A,J)$ the partial compactification of the moduli space of curves of genus $g$, obtained by adding to $\mathcal{M}_g(A,J)$ the maps of the form $q\circ u:\Sigma_g\longrightarrow Z$, with $u\in \mathcal{M}_g(A,J;\tilde{Z})$. The space $\mathcal{N}(A,J;U)$ is the set of all elements in $\mathcal{N}_g(A,J)$ with values in $U$.
\end{dfn}

\subsection{Transversality for closed holomorphic curves}

We would like to prove that the moduli space is a smooth manifold. We consider here the partial compactification $\mathcal{N}_g(A,J)$, as in \Cref{defpartialcompactification}, and prove that that is smooth, under some assumption. In particular, we prove smoothness for (the partial compactification of) the moduli space of genus $0$ maps in $4$-dimensional log-symplectic manifolds.\\
By the maximum principle (\Cref{leaf}), one has
\begin{equation}
\NN_g(A,J)=\NN_g(A,J;U)\sqcup \M_g(A,J;X\mi U)
\end{equation}
Since regularity can be expected for somewhere injective curves, let us introduce a notation. We denote by
\[
\NN^*_g(A,J):=\NN^*_g(A,J;U)\sqcup \M^*_g(A,J;X\mi U)
\]
where $\M^*_g(A,J;X\mi U)$ is the subset of somewhere injective curves, and $\NN^*_g(A,J;U)$ is the set of curves of the form $q\circ u$ with $u$ somewhere injective.\\
We know from the general theory that the curves in $\M^*_g(A,J;X\mi U)$ are Fredholm regular for a generic choice of almost complex structure (\Cref{transversalityclosed}). The regularity of curves in a neighbourhood of the singular locus is more delicate: we observe in the following example that we cannot hope for Fredholm regularity for positive-genus holomorphic curves.
\begin{eg}
Consider $\R\times S^1\times \Sigma_g$, and the simple map $u:\Sigma_g\longrightarrow \R\times S^1\times \Sigma_g$. This is holomorphic for appropriate choices of cylindrical complex structures. The pull-back of the tangent bundle is $\underline{\C}\times T\Sigma_g$. The $\overline{\partial}$ operator splits, because of integrability of the complex structure. This operator has always a non-trivial cokernel by the Riemann-Roch formula, unless $g=0$.
\end{eg}
So let us content ourselves to consider the genus $0$ case. The first step is to show that in some cases one can deduce the Fredholm regularity from the symplectic leaves. Assume we are in a neighbourhood $U$ of the singular locus $Z$, and $u:S^2\longrightarrow U$ is of the form $u=(\lambda_0,\,[t_0,\,v])$ with $v:S^2\longrightarrow \tilde{F}_{t_0}$. Then $u^*TX=\underline{\C}\oplus v^*T\tilde{F}$. 
\begin{coro}\label{regF}
If the complex structure on $\tilde{F}_{t_0}$ is integrable, $u$ is regular if and only if $v$ is regular.
\end{coro}
\begin{proof}
This follows immediately from \cite{mcsa} Chapter 3.3.
\end{proof}
This is enough to prove that all simple holomorphic spheres are generically regular in dimension $4$. Indeed, if $Z$ has a component $Z_0$ for which the symplectic fibers are not spheres, then we know that there are no holomorphic spheres with values in the correponding component $U_0$ of $U$. Hence can restrict to the case where $Z$ is a collection of $S^1\times S^2$'s. By the formula (\ref{isomodulinearZ}) we only need to prove that 
\begin{itemize}
\item $\M^*(A,J;\tilde{Z})$ is smooth
\item the action of $\Z_2$ on $(-\varepsilon,\,\varepsilon)\times\M_g(A,J;\,\tilde{Z})$ is free
\end{itemize}
The first item follows from generic transversality, and \Cref{regF}. Actually, what is true is that the only somewhere injective maps of a sphere into a sphere are the biholomorphisms, which all represent the same homology class $A$. For that homology class, $\M^*(A,J;\tilde{Z})=S^1$. Let us turn the second item into a lemma.
\begin{lemma}\label{freeaction}
$\Z_2$ acts freely on $(-\varepsilon,\,\varepsilon)\times\M_g(A,J;\,\tilde{Z})$.
\end{lemma}
\begin{proof}
Assume $(-1)\cdot(\lambda,\,v)=(\lambda,\,v)$. Then there exists a biholomorphism $\varphi:S^2\longrightarrow S^2$ such that
\[
(\lambda,\,v)=(-\lambda,\,\sigma\circ v\circ \varphi)
\]
or equivalently
\[
\sigma \circ v= v\circ \varphi
\]
It is enough to show that $\sigma$ induces a fixed-point free map on the leaf space ($S^1$) of $\tilde{Z}=S^1\times S^2$. That has to be the case because if there exists a $t\in S^1$ such that $\sigma(t,\,z)=(t,\,\tau(z))$, for a symplectomorphism $\tau$. $\tau$ is Hamiltonian, hence has fixed points. Since $\sigma$ needs to be free, that's a contradiction.
\end{proof}
\begin{coro}\label{localmodelN}
$\NN^*(A,J;U)$ is smooth and diffeomorphic to $(-\varepsilon,\,\varepsilon)\times_{\Z_2} \M_g(A,J;\,\tilde{Z})$
\end{coro}
The proof actually shows that $\NN^*(A,J;U)$ is a cylinder when $X$ is orientable and a M\"obius band when $X$ is not orientable.

\begin{thm}\label{4dtrans}
Let $(X,\,Z,\,\omega)$ be a log-symplectic manifold of dimension $4$. There is a comeager set of cylindrical complex structures $\J_\text{reg}$ such that for all $J$ in $\J_\text{reg}$ all simple holomorphic maps $u:S^2\longrightarrow X$ are Fredholm-regular. In particular the partial compactification of the moduli space of simple curves $\NN^*_0(A,J)$ is a smooth manifold of dimension $2c_1(A)-2$.
\end{thm}
\begin{proof}
Consider a tubular neighbourhood $U\cong (-\varepsilon, \varepsilon)\times_{\Z_2} Z$.
We already know that for a generic choice of $J$ all the spheres with values in $X\mi U$ are Fredholm regular, so we only need to look at $U$. Let $F$ be the ($2$-dimensional) fiber. If $F$ has genus $g>0$, then the theorem is just a consequence of \Cref{regaspherical1}.
If $F$ has genus $g=0$, we need to study the regularity of simple curves with values in $U\mi Z$. Since all complex structures on $F$ are integrable, we can apply the regularity theorem \Cref{regF}. The only simple holomorphic map has Chern number $2$, which is bigger than $-1$. Hence it is Fredholm regular. To conclude, as a consequence of \Cref{freeaction} we obtain that $\NN^*(A,J;U)$ is smooth also at points corresponding to $Z$-valued maps.
\end{proof}

\subsection{Ruled surfaces}

We noted in \Cref{obstruction} that for a log-symplectic $4$-manifold the presence of a singular component diffeomorphic to $S^1\times S^2$ forces all the singular locus to be a union of $S^1\times S^2$. This condition turns out to be even more restrictive, as it forces the entire manifold to be a blow-up of a \textit{ruled surface}, i.e. a blow-up a $4$-manifold supporting a fibration $X\longrightarrow B$ over a surface, with symplectic fiber $S^2$. This result part of what we called \Cref{B} in the introduction, and is analogous to what McDuff proves in \cite{mcduffruled}. There it is proved that in a symplectic $4$-manifold the existence of a symplectically embedded sphere with trivial normal bundle is equivalent to the manifold being ruled (up to blow-up). The key fact is the following refined compactness theorem from \cite{wendlruled}.
\begin{thm}\label{minusonecurves}
Let $(M,\,\omega)$ be a closed symplectic $4$-manifold. Let $A$ be a homology class with $A\cdot A=0$. Let $J$ be a generic complex structure, and consider a sequence of embedded $J$-holomorphic spheres representing the class $A$. Then a convergent subsequence converges to either an embedded sphere (in the same class $A$), or to a nodal curve with two irreducible components, in classes $A_1$ and $A_2$, with self-intersection $-1$, and intersecting each other positively at exactly one point. Moreover the moduli space of such nodal curves is compact and $0$-dimensional.
\end{thm}
The proof of this result is based on an analysis of the virtual dimension, combined with the adjunction formula (in particular the fact that the index determines whether the curve is embedded). The genericity of $J$ is needed in order to prevent negative-index holomorphic curves to appear. The exact same proof applies to sequences of holomorphic curves in $X\mi Z$ converging to a nodal curve in $X\mi Z$. Hence, we can prove a log-symplectic version of \Cref{minusonecurves} as soon as we can prevent bubbling in a neighbourhood of $Z$.\\
Let us give a name to the $4$-manifolds with singular locus containing $S^1\times S^2$.
\begin{dfn}
A log-symplectic $4$-manifold $(X,\,Z,\,\omega)$ is called \textit{genus-$0$ log-symplectic} if one component of the singular locus is diffeomorphic to $S^1\times S^2$. 
\end{dfn}
We already know (\Cref{obstruction}) that if $(X,\,Z,\,\omega)$ is genus-$0$ then all of the components of $Z$ are $S^1\times S^2$.
Moreover if $(X,\,Z,\omega)$ contains a symplectic sphere with trivial normal bundle then $X$ is genus-$0$ (\Cref{obstruction0}), and conversely if $X$ is genus $0$ then the $S^2$-factor in the singular locus is an embedded symplectic sphere with trivial normal bundle. We first show the spheres in the singular locus are in fact all homologous (modulo double covers). In particular they all appear in the same moduli space.
\begin{prop}\label{preliminary}
Let $(X,\,Z,\,\omega)$ be a closed log-symplectic $4$-manifold. Assume there exists an embedded symplectic sphere representing a homology class $A$ with zero self-intersection. Then all components of the singular locus are diffeomorphic to $S^1\times S^2$. The double cover $\tilde{Z}$ is also diffeomorphic to $S^1\times S^2$, and for each component of $\tilde{Z}$ $[\{p\}\times S^2\subset \tilde{Z}]=A$.
\end{prop}
\begin{proof}
The proof goes as in \Cref{mcduff}. We already noted that all the components of the singular locus are $S^1\times S^2$'s, and the same is true for the components of $\tilde{Z}$. Take neighbourhood of $Z$ such that $U\cong(-\varepsilon,\,\varepsilon)\times_{\Z_2}( S^1\times S^2)$; the complement of $\{0\}\times_{\Z_2}( S^1\times S^2)$ is just $(-\varepsilon,\,0)\times S^1\times S^2$. Pick a cylindrical complex structure that makes the embedded sphere $S$ into a $J$-holomorphic sphere, and such that $\{(\lambda,\,t)\}\times S^2\subset (-\varepsilon,\,0)\times S^1\times S^2$ is holomorphic. By automatic transversality we can assume that $J$ is generic. The moduli space of simple spheres in the class $A$, together with the spheres in the singular locus, is then a smooth manifold (\Cref{4dtrans}). It has dimension $2$ by the adjunction formula. Let $N$ be the set of points belonging to a reducible nodal curve in the class $A$. The evaluation map
\[
\text{ev}:\M_{0,1}(A,J)\cap \text{ev}\1((X\mi Z)\mi N )\longrightarrow (X\mi Z)\mi N
\]
is a proper map. By positivity of intersection and the fact that $A\cdot A=0$, there is at most one $A$ curve through each point. The degree of the evaluation is $1$, hence the evalutation is a bijection (it is in fact a diffeomorphism). In particular there exists an $A$-curve through each point of $U\mi Z$. By the maximum principle it has to be contained in $\{(\lambda,\,t)\}\times S^2\subset \tilde{Z}$, and hence coincide with $\{(\lambda,\,t)\}\times S^2$.
\end{proof}

\begin{thm}\label{minusonecurveslog}
Let $(X,\,Z,\,\omega)$ be a closed log-symplectic $4$-manifold. Let $A$ be a homology class with $A\cdot A=0$. Let $J$ be a generic complex structure, and consider a sequence of embedded $J$-holomorphic spheres representing the class $A$. Then a convergent subsequence converges to either an embedded sphere (in the same class $A$), or to a nodal curve with two irreducible components, in classes $A_1$ and $A_2$, with self-intersection $-1$, and intersecting each other positively at exactly one point. Moreover the moduli space of such nodal curves is compact and $0$-dimensional, and the nodal curves only appear in the symplectic locus.
\end{thm}
\begin{proof}
We know from the above \Cref{preliminary} that the singular locus consists of $S^1\times S^2$'s in the same class. Hence, by \ref{localmodelN}, the moduli space of curves with values in $U$ looks like $(-\varepsilon,\,\varepsilon)\times_{\Z_2} \M_0(A,J; \tilde{Z})\cong (-\varepsilon,\,\varepsilon)\times_{\Z_2} S^1$. Hence no bubbling can occur in $U$. The statement now follows from \Cref{minusonecurves}
\end{proof}
This result, together with the fact that the nodal singularities can be seen as Lefschetz singularities (\cite{wendlruled}) implies the following characterization.
\begin{prop}\label{logruled}
Let $(X,\,Z,\,\omega)$ be a closed genus-$0$ log-symplectic $4$-manifold. Then $X$ supports a Lefschetz fibration $X\longrightarrow B$ with fibers of genus $0$, so that $Z$ is a union of regular fibers. Moreover, the singular fibers have exactly one singular point, and their irreducible components are $(-1)$-curves. In particular if $X\setminus Z$ is minimal, then $X\longrightarrow B$ has no critical points.
\end{prop}
\begin{proof}{(see also \cite{wendlruled})}
The statement follows from the evaluation map 
\[
\text{ev}:\overline{\NN}_{0,1}(A,J)\longrightarrow X
\]
being a diffeomorphism. The base of the Lefschetz fibration is $\overline{\NN}_0(A,J)$. 
\end{proof}
\begin{coro}
A genus-$0$ log-symplectic $4$-manifold $X$ is diffeomorphic to a blow-up of one of the manifolds in the following list:
\begin{itemize}
    \item $S^2\times \Sigma_g$
    \item $S^2\times (\#^k\R P^2)$
    \item $S^2\widehat{\times}\Sigma_g$
    \item $S^2\widehat{\times}(\#^k\R P^2)$
\end{itemize}
where by $S^2\widehat{\times}B$ we mean the unique (up to diffeomorphism) non-trivial oriented sphere bundle over $B$.
\end{coro}
In order to conclude the proof of \Cref{B}, we need to discuss uniqueness of the log-symplectic form. It is shown in \cite{wendlruled} that any two symplectic forms supported by the same Lefschetz fibration can be deformed one into the other. The proof is linear algebraic, as the non-trivial part is to interpolate with non-degenerate forms. For this reason, the same proof holds for \textit{Lie algebroid Lefschetz fibrations} (\cite{ralphgil2}).\\
The statement is the following. 
\begin{prop}
	Let $F:A_M^4\longrightarrow A_B^2$ be a Lie algebroid Lefschetz fibration, and $\omega_0$, $\omega_1$ be symplectic forms supported by the fibration and inducing the same orientation on $A_M$. Then there is a path $\omega_s$ of symplectic forms joining $\omega_0$ and $\omega_1$.
\end{prop}
In particular, the statements applies to the fibrations constructed in \Cref{logruled}. For more details see \cite{mythesis}.

\section{Curves with boundary}\label{5}

The aim of this section is to study logarithmic holomorphic curves with boundary on the singular locus, as well as closed logarithmic holomorphic curves with non-empty singular locus. We will construct moduli spaces of the latter in a particular case, and apply our construction to the study of ruled surfaces.

\subsection{Moduli spaces of curves with boundary}

We consider here (moduli spaces of) holomorphic maps of a compact logarithmic Riemann surfaces with boundary. We know from \Cref{asympt} and \Cref{logequalpunctured} that such curves correspond to finite energy punctured holomorphic curves, as studied in SFT, and their boundary components are mapped to simple closed Reeb orbits.\\
%First of all, we note that the Reeb orbits that can be the image of the boundary of such curves need to come from the double cover $\tilde{Z}$ of the singular locus. Assume for simplicity that the boundary of $\Sigma$ is connected. 
%\begin{lemma}\label{lemmapathsdoublecover}
%Let a neighbourhood of $Z$ be of the form $(-\varepsilon,\,\varepsilon)\times_{\Z_2}\tilde{Z}$, let $q:\tilde{Z}\longrightarrow Z$ be a double cover. Let $\gamma$ be a Reeb orbit and let $u:(\Sigma,\,\partial\Sigma)\longrightarrow (X,\,Z)$ be such that $u(\partial\Sigma)=\gamma$. Then $[\gamma]\in q_*(\pi_1(\tilde{Z}))\subset \pi_1(Z)$. 
%\end{lemma}
%\begin{proof}
%Transversality at $Z$ implies that $\Sigma$ is a cylinder in a small neighbourhood of $Z$. This cylinder induces a cobordism between $\gamma$ and a loop in $\tilde{Z}$, by transversality; hence $\gamma$ is in the image of $q_*$ in homology. It is easy to show this implies that $\gamma$ is in the image of $q_*$ in homotopy.
%\end{proof}
%{\color{red}find a better place for that.}
%\begin{rk}
%\Cref{lemmapathsdoublecover} can also be viewed as a consequence of the asymptotic behaviour of finite energy punctured holomorphic curves in the completion of a compact manifold $W$. Indeed, $X\mi Z$ is the completion of a manifold $W$ with boundary $\partial W=\tilde{Z}$.
%\end{rk}
We construct the moduli spaces using this correspondence. To be more precise, fix a tubular neighbourhood of $Z$ so that the resulting stable Hamiltonian structure is Morse-Bott, and pick a cylindrical complex structure. As in \Cref{subsectionmodulisft}, one can fix a set $\mathfrak{c}$ of asymptotic constraints, and consider the space $\M_g(A,J)^\mathfrak{c}$. Applying \Cref{generictransversalitysft} one obtains immediately the following generic transversality theorem:
\begin{thm}\label{thmtransversalitylogopen}
	Let $(X,\,Z,\,\omega)$ be log-symplectic. Fix a tubular neighbourhood $U$ of $Z$ such that the induced stable Hamiltonian structure is Morse-Bott, and pick a cylindrical complex structure $J_{\text{fix}}$ on $U$. Consider the set $\mathcal{J}_{\text{fix}}$ of compatible complex structures that coincide with $J_{\text{fix}}$ on $U$. Let $\mathfrak{c}$ be a set of asymptotic constraints. Then for a generic choice of $J\in \mathcal{J}_{\text{fix}}$ all curves in  $\M_g(A,J)^\mathfrak{c}$ are Fredholm regular, so that the moduli space is a smooth manifold of dimension $(n-3)(2-2g-\#\pi_0(\partial\Sigma))+2c_1^\tau(A)+\mu(\mathfrak{c})$.
\end{thm}
\begin{proof}
	All curves are embedded near the boundary, hence they are somewhere injective (or, equivalently, simply covered). Moreover, by the maximum principle all curves must intersect $X\mi U$, which is the locus where the complex structure is allowed to vary. Hence \Cref{generictransversalitysft} applies, yielding the result.
\end{proof}
In the same way the SFT compactness theorem \ref{thmsftcompactness} applies.
\begin{thm}\label{thmlogsftcompactness}
	With the same notation as above, the moduli space $\M_g(A,J)^\mathfrak{c}$ is compactified by the space $\overline{\M}_g(A,J)^\mathfrak{c}$ of stable holomorphic buildings.
\end{thm}
In the special case when no two boundary components are mapped to the same component of the singular locus, the geometric setup allows to rule out certain buildings.
\begin{prop}\label{onepuncturecompactness}
	Let $u_k$ be a sequence of holomorphic curves with asymptotic constraint $\mathfrak{c}$ so that no two boundary components are mapped to the same component of the singular locus. Let $u_\infty$ be a limiting holomorphic building. Then all the levels other than the main one consist of disjoint unions of nodal curves with exactly two punctures (one for each component of the singular locus), asymptotic to simple Reeb orbits. The main level must be a nodal curve with as many punctures as the $u_k$'s, with at most one puncture asymptotic to each component of $Z$.
\end{prop}
\begin{proof}
	For each component of the singular locus, the top level is a nodal curve with a single positive puncture asymptotic to a single simple curve. The sum of the negative punctures must be homotopic (in $Z$) to a simple curve. In particular, their projection to $S^1$ need to be homotopic. By holomorphicity, all the negative punctures are multiples of the positive generator of $\pi_1(S^1)$, which implies there must be a single simple negative puncture. Iterating the argument one sees that it has to hold for all lower levels (main exluded), and that the punctures of the main level are in one to one correspondence with the punctures of $u_k$.
\end{proof}
%\begin{dfn}[Broken nodal curves]
%A \textit{broken nodal holomorphic curve} asymptotic to $\Gamma=\{\gamma_1,\,\dots,\,\gamma_k\}$ consists of the following data:
%\begin{itemize}
%    \item a $k$-punctured Riemann surface $(S,\,j)=\underset{a}{\sqcup}(S_a,\,j_a)$
%    \item a holomorphic map $u=(u_a)_a:\underset{a}{\sqcup} S_a\longrightarrow X\mi Z$, where none of the $u_a$'s is constant, and such that the punctures are asymptotic to simple closed Reeb orbits $\tilde{\Gamma}=\{\tilde{\gamma}_1,\,\dots,\,\tilde{\gamma}_k\}$
%    \item a set $N$ of points in $\sqcup S_a$, called \textit{nodes}, with a fixed-point-free involution $\sigma:N\longrightarrow N$, such that $u(n)=u(\sigma(n))$ $\forall\,n\in N$, and such that the singular surface obtained by glueing $n$ with $\sigma(n)$ is connected
%    \item for each $i\in\{1,\dots,\,k\}$, a set of holomorphic strips $v^i_b:\R\times [0,\,1]\longrightarrow \tilde{F}$, $b\in \{0,\dots m(i)\}$ such that $v^i_0(-\infty)=\tilde{\gamma}_i$, $v^i_b(+\infty)=v^i_{b+1}(-\infty)$, $v^i_{m(i)}(+\infty)=\gamma_i$ (where we identify a simple closed orbit with the corresponding fixed point of the monodromy map).
%\end{itemize}
%\end{dfn}
%
%\begin{thm}\label{openlogcompactness}
%Consider a Morse-Bott log-symplectic manifold. Let $S=(S_1,\,\dots,\,S_k)$ be sets of Reeb orbits, let $A$ be a relative homology class, and $J$ a cylindrical complex structure. The space $\M_g(A,J,S)$ can be compactified by adding broken nodal curves.
%\end{thm}

\subsection{Punctured spheres with $0$ self-intersection}

In this section we investigate the consequences of the existence of an embedded holomorphic punctured sphere with $0$ self-intersection in certain log-symplectic $4$-manifolds.\\

The intersection number of punctured holomorphc curves was introduced by \cite{siefring} for non-degenerate asymptotics, and by \cite{siefringwendl}, \cite{wendlpunctured} for Morse-Bott asymptotics). It is a homotopy invariant number $u\star v$ associated to pairs of holomorphic curves $u$ and $v$ in a $4$-manifold, with some asymptotic constraints. As is the case with closed holomorphic curves (see \cite{mcduffadjunction}), there is an adjunction formula relating the topology of a somewhere injective curve, the number of its singular points and its self-intersection number. \Cref{sectionintersectionpunctured} contains a brief introduction to the intersection theory of punctured holomorphic curves, with a focus on the results that we need in the present section. We also refer to \Cref{computationintersections} for the explicit computations of the intersection numbers and indices appearing in the present section.\\

Let $(X,\,Z,\,\omega)$ be an oriented log-symplectic $4$-manifold. Let $X_0$ be a component of $X\mi Z$. We know (\Cref{stablehamiltonianform}) that we can realize $X_0$ (non-uniquely) as a completion $\widehat{W}$ of a manifold $(W,\,\partial W)$ with stable Hamiltonian boundary, such that the $1$-form $\alpha$ is closed. We will study punctured spheres with $0$ self-intersection in the case when the stable Hamiltonian structure induces the trivial fibration $S^1\times \Sigma_{g_i}$ on each component $Z_i$ of $\partial W$. The goal is again to apply the results to the study of ruled surfaces (\Cref{sectionlogruled}).\\

The precise statement is the following.
\begin{thm}\label{theoremmoduliboundary0intersection}
	Let $(W,\,\partial W,\,\omega)$ be a symplectic $4$-manifold with an induced stable Hamiltonian structure on $\partial W$, of the form $(\alpha,\,\beta)$, with $d\alpha=0$. Assume that $(\alpha,\,\beta)$ induces a trivial symplectic fibration, so that $\partial W=\sqcup_i S^1\times \Sigma_{g_i}$.\\
	Let $J$ be a ``generic" cylindrical complex structure, and let $u:S^2\mi\{p_1,\,\dots,\,p_k\}\longrightarrow \widehat{W}$ be a $k$-punctured $J$-holomorphic sphere such that
	\begin{itemize}
		\item different punctures are asymptotic to Reeb orbits in different components of $[R,\,+\infty)\times \partial W$;
		\item $u\star u= 0$ as punctured holomorphic curves with unconstrained asymptotics.
	\end{itemize}
	Let $\M$ denote the component containing $u$ of the moduli space of $J$-holomorphic maps. Then $\M$ is a smooth $2$-dimensional manifold, all elements of which are embedded curves. Moreover $\M$ can be compactified to a smooth closed $2$-dimensional manifold $\overline{\M}$, by adding nodal curves consisting exactly of $2$ components: an embedded $k$-punctured sphere of self-intersection $-1$, an embedded sphere of self-intersection $-1$. The two components intersect exactly at one point.
\end{thm}
The proof will consist of an index computation, and an application of the adjunction formula (\Cref{thmadjunctionpunctured}). We will start with an arbitrary holomorphic building, and compare the index of each level with the index of $u$. The genericity of $J$ will be used to rule out all such buildings containing levels with negative index. Using the adjunction formula, we can relate the indices of the remaining buildings with their self-intersection, and their embeddedness. As a result, we will prove that the only holomorphic buildings that an appear as limits of curves homotopic to $u$ are those described in the statement.\\

A useful tool in the proof is the following lemma, which allows to tell whether a cylinder is trivial.
\begin{lemma}\label{lemmarecognisecylinders}
	Let $Z$ be a symplectic mapping torus, isormorphic to $(S^1\times \Sigma_g,\,cdt,\,\text{vol}_{\Sigma_g})$, and let $J$ be a compatible almost complex structure on $\Re\times Z$. Let $u:\C^*\longrightarrow \R\times Z$ be a finite energy holomorphic cylinder. Then, if $g>0$, $u$ is a trivial cylinder. If $g=0$, then the relative Chern class $c_1(u)$, defined with respect to the isomorphism $Z\cong S^1\times S^2$, is a non-negative even number, and $c_1(u)=0\Leftrightarrow u$ is a trivial cylinder.
\end{lemma}
\begin{proof}
	Fix an isomorphism $Z=S^1\times\Sigma_g$ as symplectic mapping tori, and write $u=(u_1,\,u_2)$. The map $u_2:\C^*\longrightarrow\Sigma_g$ is holomorphic and has finite energy, thus extends to $u_2:S^2\longrightarrow\Sigma_g$. If $g>0$, this means that $u_2$ is a constant.\\
	
	If $g=0$, one has $u_2:S^2\longrightarrow S^2$, and $c_1(u_2)=2\text{deg}(u_2)\geq 0$ because $u_2$ is orientation preserving. In particular $u_2$ is constant if and only if $\text{deg}(u_2)=0$ (by holomorphicity). The fact that $c_1(u)=c_1(u_2)$ implies the statement.
\end{proof}

\begin{proof}[Proof of \Cref{theoremmoduliboundary0intersection}]
	First of all, using the adjunction formula one computes that $\text{ind}(u)=2$ (\Cref{indexfromintersection}). Since $J$ is a generic complex structure, there are no $J$-holomorphic curves of negative index.\\
	Let us assume now that $u:S^2\mi \{p_1,\,\dots,\,p_k\}\longrightarrow \widehat{W}$ is $J$-holomorphic; let $\gamma_i$ be the orbit to which $p_i$ is asymptotic. Consider the moduli space $\M:=\M_0(A,J)^\mathfrak{c}$ of $J$-holomorphic punctured spheres, with punctures $p_1,\,\dots,\,p_k$ asymptotic to simple orbits in the same components as $\gamma_1,\,\dots,\,\gamma_k$, respectively. By the transversality theorem \ref{thmtransversalitylogopen}, $\M$ is a smooth $2$-dimensional manifold.\\
	
	By \Cref{onepuncturecompactness}, we can compactify $\M$ by adding holomorphic buildings $\uu_\infty=(\uu_0,\uu_i)$, where $\uu_0$ is a nodal punctured sphere in $\widehat{W}$, and $\uu_i$, $i>0$ is a nodal cylinder in $\R\times \partial W$. In particular we can write each $\uu_i$ as a collection of curves: $\uu_i=(u_{i,0},\,u_{i,1},\,\dots,\,u_{i,m_i})$, where $u_{i,j}$ are spheres for $j>0$, and $u_{i,0}$ is either a $k$-punctured sphere (when $i=0$) or a cylinder (when $i>0$). We need to show that:
	\begin{itemize}
		\item $m_i=0$ for all $i>0$;
		\item $u_{i,0}$ is the trivial cylinder for all $i>0$;
		\item $m_0=1$;
		\item $\uu_0=(u_0,\,v_1)$, is such that $u_0\star u_0=-1=v_1\cdot v_1$, and $u_0\cdot v_1=1$.
	\end{itemize}
	We will do this by index counting, assuming that there are no holomorphic curves of negative index (which is true due to our assumption that $J$ is generic).\\
	
	Denote by $N_{i,j}$ the number of nodes of the curve $u_{i,j}$, and let $S_{i,j}$ be its domain. Let $S_i$ be the (topological) connected sum of the $S_{i,j}$'s at the nodes. Then \begin{equation}
	\chi(S_i)=\sum_{j\geq0}(\chi(S_{i,j})-N_{i,j})
	\end{equation} 
	and 
	\begin{equation}
	2-k=\chi(S^2\mi \{p_1,\,\dots,\,p_k\})=\sum_{i,j}\chi(S_i)=\chi(S_0)
	\end{equation}
	(the last equality follows from $S_i$ being a cylinder for $i>0$). Denote by $A_i$ the relative homology class of $u_{i,0}$, and $B_{i,j}$ the homology class of $u_{i,j}$.\\
	
	Define the following numbers:
	\begin{itemize}
		\item
		$\text{ind}(\uu_\infty):=(n-3)\chi(\dot{\Sigma})+2c_1^\tau(A)+\mu^\tau(\mathfrak{c})=-\sum_{i,j}(\chi(S_{i,j})-N_{i,j})+\sum_i 2c_1^\tau(A_i)+\sum_{i,j}2c_1(B_{i,j})+\mu^\tau(\mathfrak{c})$
		\item
		$\text{ind}(\uu_0):=-\sum_{j}(\chi(S_{0,j})-N_{0,j})+2c_1^\tau(A_0)+\sum_{j}2c_1(B_{0,j})+\mu^\tau(\mathfrak{c}')$
		\item
		$\text{ind}(\uu_i):=-\sum_{j}(\chi(S_{i,j})-N_{i,j})+2c_1^\tau(A_i)+\sum_{j}2c_1(B_{i,j})+\mu^\tau(\mathfrak{c}'), \text{ for } i>0$
	\end{itemize}
	One has 
	\begin{equation}
	2=\text{ind}(\uu_\infty)=\text{ind}(\uu_0)+\sum_{i\geq 1}(\text{ind}(\uu_i)-\mu^\tau(\mathfrak{c}'))=\text{ind}(\uu_0)+\sum_{i\geq 1}(\text{ind}(\uu_i)-2)
	\end{equation}
	We'd like to prove that $\text{ind}(\uu_0)$ and $\text{ind}(\uu_i)-2$ are all non-negative.\\
	
	It is immediate to check that 
	\begin{equation}
	\text{ind}(\uu_0)=\text{ind}(u_{0})+N_{0,0}+\sum (\text{ind}(u_{0,j})+N_{0,j})
	\end{equation}
	and all the summands are non-negative. Similarly, for $i>0$, 
	\begin{equation}
	\text{ind}(\uu_i)=(\text{ind}(u_{i,0})+N_{i,0}-2)+\sum_{j\geq 1}(\text{ind}(u_{i,j}+N_{i,j}))\geq \sum_{j\geq 0}(N_{i,j})-2
	\end{equation}
	
	We claim that $\text{ind}(\uu_i)-2\geq 0$.\\ 
	First of all, \Cref{lemmarecognisecylinders} ensures that $c_1(A_i)\geq 0$ and is even whenever $i>0$. Moreover, also $c_1(B_{i,j})\geq 0$, because by the maximum principle $u_{i,j}:S^2\longrightarrow \{\star\}\times S^2$. Finally, $c_1(A_i)=0$ if and only if $A_i$ represents a trivial cylinder (in which case $N_{i,0}>0$, by stability), while $c_1(B_{i,j})=0$ if and only if $B_{i,j}$ represents a constant bubble (in which case $N_{i,j}>2$, again by stability).\\
	
	Now, \begin{equation}
	\text{ind}(\uu_i)-2=2c_1^\tau(A_i)+\sum_{j}2c_1(B_{i,j})-\sum_{j}(\chi(S_{i,j})-N_{i,j})\geq 0+\text{ind}(u_{i,j})+\sum_j N_{i,j}\geq 0
	\end{equation}
	which proves the claim.\\
	
	In fact, we can prove that $\text{ind}(\uu_i)-2\geq 3$, which contradicts the estimate $2\geq \text{ind}(\uu_i)-2$. Indeed, 
	\begin{equation}
	\text{ind}(\uu_i)-2=2c_1^\tau(A_i)-(0-N_{i,0})+\sum_j 2c_1(B_{i,j})-(2-N_{i,j})
	\end{equation}
	It was already observed that the stability condition implies $2c_1^\tau(A_i)+N_{i,0}\geq 1$. If there are no bubbles, then $2c_1^\tau(A_i)\geq 4$. If there is a non-constant bubble, then $2c_1(B_{i,j})-(2-N_{i,j})\geq 4-2+1=3$ (a bubble has at least one node). If all bubbles are constant, then there must be at least $6$ nodes (by stability, and the fact that nodes come in pairs), hence $\ind(\uu_i)-2\geq 6-2=4$. All cases contradict $2\geq \text{ind}(\uu_i)-2$, implying that there cannot be any levels other than the main one.\\
	
	We are left to prove that the main level is either an embedded curve with $0$ self-intersection, or it consists of two components, one of which is a punctured sphere and the other a sphere, intersecting each other in a point, and with self-intersection $-1$. We know that $2=\ind(\uu_\infty)=\ind(\uu_{0,0})=\ind(u_0)+N_{0,0}+\sum_j \ind(u_{0,j})+N_{0,j}$.\\
	Since the orbits to which $u_0$ is asymptotic are simple and of minimal period, $u_0$ is a simply covered curve, hence the adjunction formula applies. Moreover, the index of $u_0$ is an even number. We can distinguish two cases:
	\begin{itemize}
		\item $\ind(u_0)=2$. Then there are no nodes, and no bubbles. Since $u_0$ is a simple curve, limit of curves with $0$ self-intersection, it has $0$ self-intersection. The adjunction formula then implies that it is embedded.
		\item $\ind(u_0)=0$. Then the only possibility is that there is a single bubble, say $v=u_{0,1}$, such that $\ind(v)=0$, and $N_{0,0}=N_{0,1}=1$. Moreover, $v$ is simply covered. Assume it is a $k$-fold cover of a simple curve $v'$, then $v'$ has to be a sphere. Further, $0=\ind(v)=2c_1(kB')-2=\ind(v')+2(k-1)c_1(B')$, where $[v]=B$, and $[v']=B'$. Since $\ind(v')\geq 0$, $c_1(B')\geq 1$, so that one needs to have $k=1$. 
	\end{itemize}
	Now, we would like to prove that $v\cdot v=u_0\star u_0=-1$, $v\cdot u_0=1$. We know (\Cref{thmintersectioneigenvalues}) that $u\star u=u\bullet_\tau u+i^\tau_\infty(u)$, where the first term is topological, and consists of the number of intersections of $u$ with a slight perturbation of $u$ along the direction of the trivialization $\tau$, and $i^\tau_\infty(u)$ (the number of ``hidden intersections at infinity"), only depends on the asymptotic trivialization and on the Reeb orbits. We know that the pair $(u_0,\,v)$ is a limit of curves with $0$ self-intersection, and that $v$ is closed. Hence, if $u$ is a curve in the moduli space, one has 
	$u\star u=0=(u_0+v)\bullet_\tau (u_0+v)+i^\tau_\infty(u)=u_0\bullet_\tau u_0+i^\tau_\infty(u_0)+v\cdot v+2u_0\cdot v=u_0\star u_0+v\cdot v+2$.
	Denote by $\text{sing}(w)$ the contribution of the singularity of a curve $w$ to the self intersection number. By the adjunction formula, $w\star w=\text{sing}(w)+\frac{1}{2}\ind(w)-1$. Moreover $u_0\cdot v\geq 1$. Thus $u_0\star u_0+v\cdot v+2\geq \text{sing}(u_0)+\text{sing}(v)\geq 0$. This implies that both curves are embedded, and their self intersection is $-1$. As a consequence, $u_0\cdot v=1$ (the single point of intersection is the node, and is transverse).\\
	
	This automatically implies that the moduli space is a smooth closed surface if the manifold is not a blow-up. If the manifold is a blow-up, the compactified moduli space in the Gromov topology is a topological surface, and it can be naturally given the structure of a smooth surface via the blow-down map.\\
	
	This concludes the proof of the theorem.
\end{proof}	
A version of \Cref{theoremmoduliboundary0intersection} for holomorphic curves with a marked point can be stated as follows.
\begin{thm}\label{thmmoduliboundary0intersectionmarked}
	Same hypotheses and notation as in \Cref{theoremmoduliboundary0intersection}. Let $\M_1$ be the component containing $u$ of the moduli space of $J$-holomorphic curves with one marked point. Then $\M_1$ is a smooth $4$-dimensional manifold, fibering over $\M$. Moreover $\M_1$ can be compactified to a manifold with boundary $\overline{\M}_1$, by adding:
	\begin{itemize}
		\item punctured nodal curves, with two irreducible components of self-intersection $-1$, intersecting in one point;
		\item holomorphic buildings of height $2$, where the main level is a curve of $\M$, and the second level is a trivial cylinder with a marked point.
	\end{itemize}
	The boundary of $\M_1$ consists of the latter elements. The evaluation map extends continuously to $\overline{\M}_1$, with values in the compactification $X$ of $\widehat{W}$.
\end{thm}
\begin{proof}
	The proof follows from \Cref{theoremmoduliboundary0intersection} and the SFT compactness theorem. Indeed, we can compactify $\M_1$ using the space of stable holomorphic buildings with one marked point. The buildings of $\overline{\M}$ together with a marked point are examples of elements of the SFT compactification. They are not all, as there might be building with levels consisting entirely of trivial cylinders, as long as they have enough marked points to be stable. In the case of the present theorem, it is possible for a level to contain exactly one trivial cylinder, if the marked point belongs to it. From this consideration the first part of the statement is proved. It is easily computed that the dimension of the set of height $2$ buildings is $3$, i.e. it has codimension $1$. The evaluation map extends continuously, by the very definition of the topology on the space of buildings.
\end{proof}
As an application, we mention the following consequence of the analysis above.
\begin{prop}\label{spacespuncturedspheres}
	Let $W$ be symplectic with stable Hamiltonian boundary isomorphic to a union of $S^1\times \Sigma_{g_i}$. Let $J$ be a cylindrical complex structure, and assume there is an embedded punctured $J$-holomorphic sphere $S$, with at most one puncture for each boundary component, and $0$ self-intersection (with unconstrained asymptotics). Then:
	\begin{itemize}
		\item[(i)] the number of punctures is equal to the number of components of $\partial W$;
		\item[(ii)] there is a diffeomorphism between the moduli space of curves homotopic to $S$, and $\Sigma_{g_i}$. In particular $g_i=g_j$ for all $i,\,j$.
	\end{itemize}
\end{prop}
\begin{proof}
	\textit{Part} $(i)$. Assume first that $W$ be minimal. View the domain of $S$ as a compact surface with boundary $(\Sigma,\,\partial \Sigma)$, with a logarithmic holomorphic structure. Consider the space $\M_1$ of holomorphic curves homotopic to $S$, with one marked point in $\Sigma$. By \Cref{thmmoduliboundary0intersectionmarked} and the minimality assumption this is a compact smooth manifold with boundary. 
	One can naturally compactify $\widehat{W}$ to a log-symplectic manifold $Y$ with boundary, such that $\partial Y$ is the singular locus. Since $\pi_0(\partial\M_1)=\pi_0(\partial \overline{S})$, by assumption $\text{ev}_*:\pi_0(\partial\M_1)\longrightarrow \pi_0(\partial Y)$ is injective. Hence, if it is surjective the first part of the statement is true.\\
	
	Consider the moduli space $\M_1$ of punctured $J$-holomorphic spheres homotopic to $S$, and its compactification $\overline{\M}_1$. By \Cref{thmmoduliboundary0intersectionmarked}, $\overline{\M}_1$ is a compact manifold with boundary, and the evaluation map $\text{ev}:\M_1\longrightarrow Y$ is continuous, sends the boundary to the boundary, and is smooth in the interior. If $\text{ev}$ is surjective, then the number of boundary components of $S$ needs to be equal to the number of boundary components of $Y$.\\ 
	
	Consider the doubles of $\M_1$ and $Y$, and extend the evaluation map in the obvious way. Its degree is well defined, and needs to be non-zero: for each point in the image of the evaluation, there is a unique curve in $\M$ mapping to it.\\
	
	In the nonminimal case one can run a similar argument noting that the limiting buildings form a codimension-$2$ space. Thus one can proceed as in \Cref{mcduff}.\\
	
	\textit{Part} $(ii)$. For each component of the boundary, one can define a map $r_i:\M\longrightarrow \mathcal{R}_i$, from the moduli space of curves homotopic to $S$ to the set $\mathcal{R}_i\cong\Sigma_{g_i}$ of simple unparametrized Reeb orbits in the $i$-th component of $\partial X$. The map assigns to each curve the Reeb orbit to which its $i$-th puncture is asymptotic. We will show that for all $i$ the map $r_i$ is a homeomorphism.\\
	
	The map $r_i$ is injective. Indeed, assume that there are two curves in $\M$ with a puncture asymptotic to the same Reeb orbit $\gamma$. Then their intersection, as curves with constrained asymptotic orbit $\gamma$ and the other punctures unconstrained, is equal to $-1$ (this follows from \Cref{comparisonconstrainedunconstrained}). Hence the curves must coincide, by positivity of intersection.\\
	
	The map $r_i$ is a local diffeomorphism. The tangent space to $\M$ at embedded curves consists of holomorphic sections of the normal bundle; the image of the differential is the restriction of the sections to the boundary. A non-zero section that vanishes at the boundary produces a new holomorphic curve with the same asymptotic orbit, which contradicts the injectivity. 
	If the manifold is minimal, this is enough to conclude that $r_i$ is a diffeomorphism. Since $\mathcal{R}_i\cong \Sigma_{g_i}$, and all boundary components are reached by holomorphic curves, we get $g_i=g_j$ for all $i,\,j$.\\
	In general, the same proof shows that the map $r_i$ is a homeomorphism. To turn it into a diffeomorphism we can define a smooth structure on the compactification by declaring the blow-down map to be a diffeomorphism, as in \Cref{theoremmoduliboundary0intersection}.
\end{proof}

\subsection{Logarithmic ruled surfaces}\label{sectionlogruled}

We use the results from the previous section to show that in certain log-symplectic $4$-manifolds, the presence of a holomorphic sphere with $0$ self-intersection implies that the manifold is (topologically) ruled, up to blow-up, proving \Cref{C} (\Cref{loglogruled}). The main ingredient is the construction of a smooth moduli space of log-holomorphic spheres.\\

Let us start with some observations on the combinatorics of the singular locus of a log-symplectic sphere. First of all, to any pair $(X,\,Z)$ consisting of a manifold and a codimension-$1$ submanifold (in particular to the pair determined by a log-symplectic structure), one can attach a finite planar unoriented graph $\Gamma(X,\,Z)$. The vertices are the components of $X\mi Z$. Each component of $Z$ determines an edge, joining the vertices corresponding to the components of $X\mi Z$ that it bounds. Here are some $2$-dimensional pictures.
\begin{figure}[h]
	\centering
	\begin{subfigure}{0.5\columnwidth}
		\centering
		\includegraphics{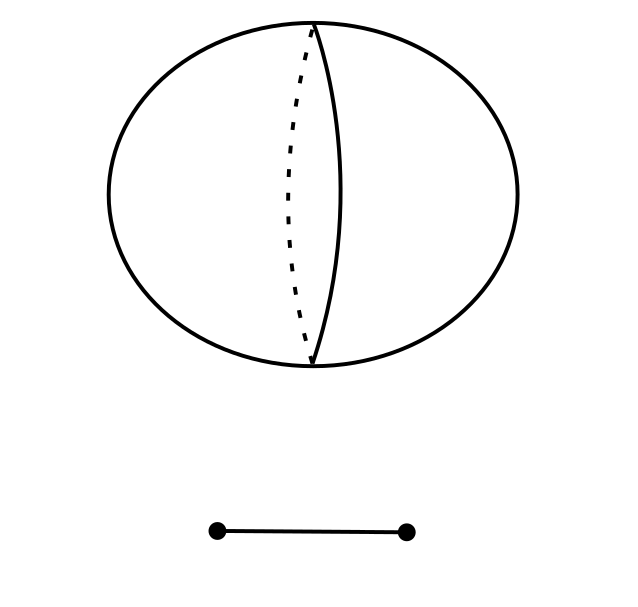}
	\end{subfigure}%
	\begin{subfigure}{0.5\columnwidth}
		\includegraphics{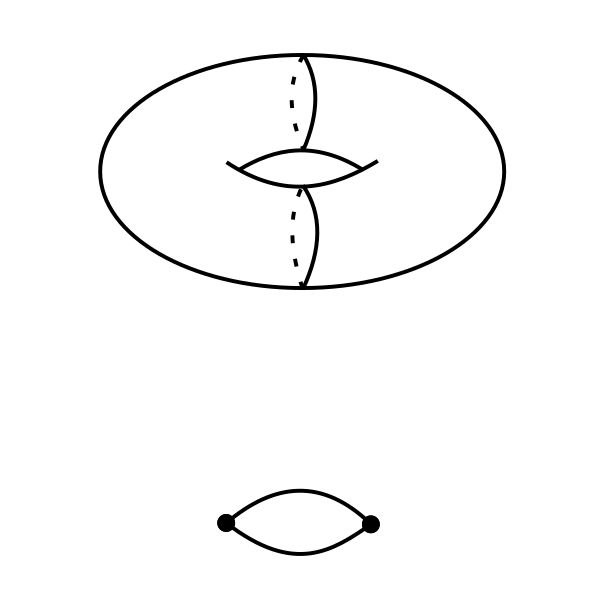}
	\end{subfigure}
	\caption{Manifolds with hypersurface, and corresponding graph}
\end{figure}

\begin{figure}[h]
	\centering
	\begin{subfigure}{0.5\columnwidth}
		\centering
		\includegraphics{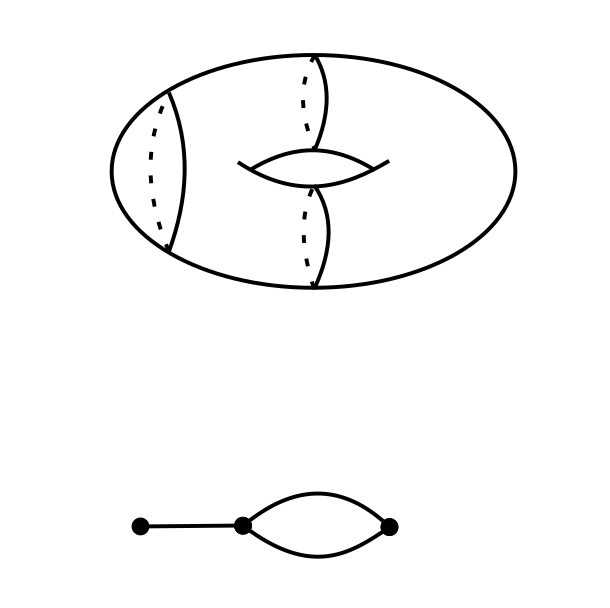}
	\end{subfigure}%
	\begin{subfigure}{0.5\columnwidth}
		\includegraphics{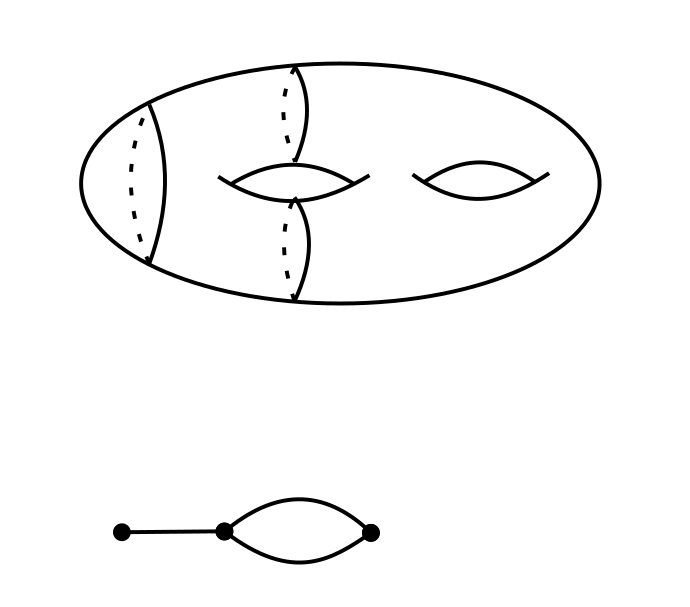}
	\end{subfigure}
	\caption{Different manifolds may give rise to the same graph}
\end{figure}

Clearly $\pi_0(\Gamma(X,\,Z))=\pi_0(X)$, in particular the graph is connected if and only if $X$ is. Let us introduce some terminology that is going to be useful.
\begin{dfn}
	Let $\Gamma$ be a finite connected graph. We call a \textit{circle} an edge with both endpoints on the same vertex. A \textit{loop} is a connected subgraph, with no circles, such that each vertex is connected to exactly two edges. We
	say that a vertex is an \textit{extreme} if all the edges connected to it, except at most one, are circles. We say that a graph is \textit{$1$-connected} if it contains no loops, and \textit{orientable} if it contains no circles.
\end{dfn}
\begin{lemma}
	Let $(X,\,Z,\,\omega)$ be log-symplectic. Then $X$ is orientable if and only if $\Gamma(X,\,Z)$ is orientable. 
\end{lemma}
\begin{proof}
	If $X$ is orientable, then $Z=\{f=0\}$ for some globally defined function vanishing linearly. $X\mi Z$ splits as $\{f>0\}\cup \{f<0\}$, and each component of $Z$, being coorientable, bounds both a component of $\{f>0\}$ and a component of $\{f<0\}$, which are necessarily different.\\
	If $X$ is not orientable, then at least one component $Z_i$ of $Z$ is not coorientable. Then there exists a neighbourhood $U_i$ of $Z_i$ such that $u_i\mi Z_i$ is connected. Hence the edge of $\Gamma(X,\,Z)$ corresponding to $Z_i$ has both extremes on the same vertex.
\end{proof}
\begin{figure}[t]
	\centering
	\includegraphics{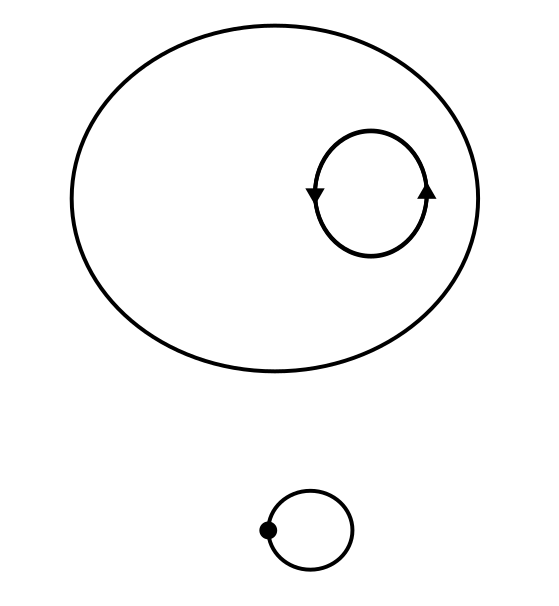}
	\caption{A circle with a single vertex, representing $\R P^2$ with the nontrivial loop as singular locus.}
	\label{figrp2}
\end{figure}
\begin{lemma}\label{lemma2extremes}
	A $1$-connected graph has at least $2$ extremes.
\end{lemma}
\begin{proof}
	The statement and proof are insensitive of circles, we might assume there are none. Pick any vertex $v_0$. Construct a sequence $v_k$ simply by selecting at each step a vertex (different from $v_k$) which is connected to $v_{k-1}$ by an edge. There are two possibilities. Case one, the process must stop, as we get to a vertex connected by only one edge. In this case such vertex is an extreme. Case two, the process can continue indefinitely. But then since the graph is finite, there must be a loop, so the graph is not $1$-connected.\\ 
	In order to find the second extreme, one can repeat the process above, taking as $v_0$ the extreme found with the procedure above. The process must end to a different extreme than $v_0$.
\end{proof}
For surfaces, we can reconstruct $(\Sigma,\,\sigma)$ given $\Gamma(\Sigma,\,\sigma)$ and the genus of the component associated to each vertex. In particular we have the following lemma.
\begin{lemma}\label{lemma1connected}
	Let $(\Sigma,\,\sigma)$ be a surface with a codimension-$1$ submanifold. If $\Sigma$ is a sphere, then $\Gamma(\Sigma,\,\sigma)$ is orientable and $1$-connected. If $\Sigma$ is $\R P^2$, then $\Gamma(\Sigma,\,\sigma)$ has exactly one circle, and is $1$-connected. %Conversely, if $\Sigma\mi\sigma$ has genus $0$ (i.e. is a union of punctured spheres) then $\Gamma(\Sigma,\,\sigma)$ orientable and $1$-connected implies that $\Sigma$ is a sphere.
\end{lemma}
\begin{proof}
	If $\Sigma$ is a sphere then every component of $\Sigma\mi \sigma$ is a punctured sphere, while if $\Sigma\cong \R P^2$, then $\Sigma\mi \sigma$ consists of punctured spheres and punctured M\"obius bands. Either way, a loop in $\Gamma(\Sigma,\,\sigma)$ would determine a(t least one) element of infinite order in $\pi_1(\Sigma)$ (see \Cref{figtorusloop}), which cannot happen. Finally, adding a circle to the graph corresponds with taking connected sum with $\R P^2$, and this can be done exactly once by the classification of surfaces.
\end{proof}
\begin{figure}[t]
	\includegraphics{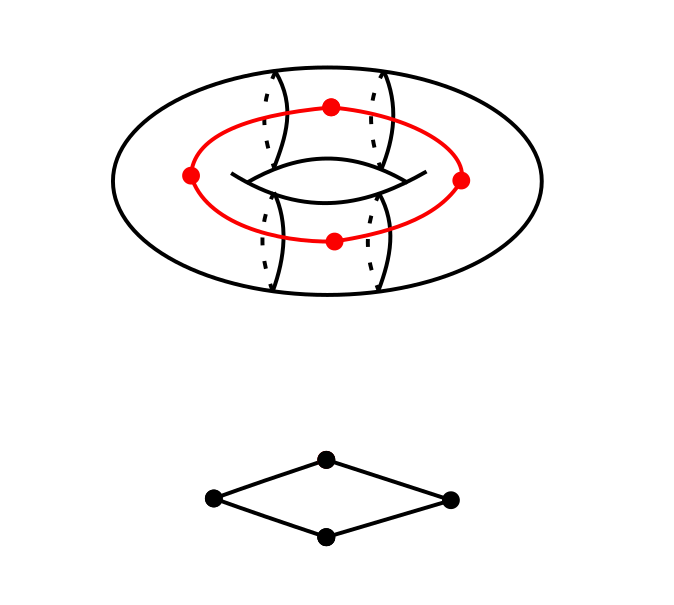}
	\centering
	\caption{A loop in the graph producing an infinite order element in $\pi_1$}
	\label{figtorusloop}
\end{figure}
\begin{dfn}
	Let $(X,\,Z,\,\omega)$ be a log-symplectic $4$-manifold. A closed embedded surface $S\subset X$ is \textit{nice} if
	\begin{itemize}
		\item $S\pitchfork Z$;
		\item the map $\pi_0(S\cap Z)\longrightarrow \pi_0(Z)$ induced from the inclusion is injective;
		\item the graph $\Gamma(S,\,S\cap Z)$ is $1$-connected, and has at least one edge.
	\end{itemize}
\end{dfn}
\begin{lemma}
	Embedded spheres and embedded real projective planes, intersecting $Z$ transversely, and intersecting each component of $Z$ in a connected set, are \textit{nice}. 
\end{lemma}
\begin{proof}
	This is a direct consequence of \Cref{lemma1connected}.
	%{\color{red}if the lemma has a converse, add it here too}.
\end{proof}

\begin{thm}\label{thmmodulilogspheres}
	Let $(X,\,Z,\,\omega)$ be a log-symplectic $4$-manifold. Assume that there is a neighbourhood of $Z$ realizing $Z$ as a union of $S^1\times \Sigma_{g_i}$, with the trivial stable Hamiltonian structure. Let $J$ be a cylindrical complex structure with respect to such trivialization. Then the moduli space $\M$ of nice $J$-holomorphic spheres with $0$ self-intersection is a smooth closed $2$-dimensional manifold.
\end{thm}
\begin{proof}
	A log $J$-holomorphic curve $u:(\Sigma,\,\sigma)\longrightarrow (X,\,Z)$ determines several punctured holomorphic curves $u_i:\dot{\Sigma}_i\longrightarrow X_i$, $i=1,\,\dots,\,N$, where $\dot{\Sigma}_i$ and $X_i$ are components respectively of $\Sigma\mi\sigma$ and $X\mi Z$. In the case at hand, each $\dot{\Sigma}_i$ is a punctured sphere. By automatic transversality (\cite{wendlpunctured}), we may assume that $J$ is regular (see also). Then, each $u_i$ satisfies the hypotheses of \Cref{spacespuncturedspheres}. Thus there are moduli spaces $\M_i$, one for each component, and they are all diffeomorphic to each other. 
	
	More precisely, recall that for each component $X_i$ of $X\mi Z$, and each component $Z_{i,j}$ bounding $X_i$, one can define a map $r_{i,j}:\M_i\longrightarrow \mathcal{R}_{i,j}$, where $\mathcal{R}_{i,j}$ is the set of simple unparametrized Reeb orbits in $Z_{i,j}$. The map is defined so that $r_{i,j}(v)$ is the Reeb orbit to which the $j$-th puncture of $v$ in $X_i$ is asymptotic; we showed in \Cref{spacespuncturedspheres} that $r_{i,j}$ is a diffeomorphism.\\
	
	In order to glue the moduli spaces $\M_i$, we exploit our understanding of the graph $\Gamma(\Sigma,\,\sigma)$. Consider first the case $\Sigma\cong S^2$. First of all we know that to each vertex there corresponds a punctured holomorphic sphere, and a corresponding moduli space. Denote $\M_x$ the moduli space corresponding to the vertex $x$. Since there are no loops, each pair of vertices determines at most one edge. To the edge $xy$ between vertices $x$ and $y$ one can associate the manifold $R_{xy}$ of Reeb orbits on the corresponding component of $Z$. Moreover, to each edge one can associate two diffeomorphisms $r_{xy}:\M_x\longrightarrow R_{xy}$ and $r_{yx}:\M_y\longrightarrow R_{xy}$.\\
	
	By \Cref{lemma2extremes} we know that there exists one extreme, which for $S^2$ has no circles. We can fix one of the extremes, say $x_0$, and partition the set of vertices by their distance from $x_0$, i.e. the number of edges that it takes to reach $x_0$. This is well defined as there are no loops. In particular, for each vertex there is exactly one set of edges connecting it to $x_0$.\\
	
	For each $x$, let $x_0x_1,\,x_1x_2,\,\dots,\,x_i x$ be the sequence of edges joining $x$ with $x_0$. Define the map $r_x=r_{x_0x_1}\1\circ r_{x_1x_0}\circ r_{x_1x_2}\1\circ r_{x_2x_1}\circ\dots\circ r_{x x_i}:\M_x\longrightarrow \M_{x_0}$ ($r_{x_0}=\text{id}$). It is a diffeomorphisms for each vertex $x$.\\ 
	
	Finally, denoting with $V(\Sigma,\,\sigma)$ the set of vertices of $\Gamma(\Sigma,\,\sigma)$, we can define the moduli space as $\M=\{(v_x)_{x\in V(\Sigma,\,\sigma)}:\,v_x\in\M_x,\,r_x(v_x)=v_{x_0}\}$. The map $(v_x)_x\mapsto v_x$ induces a diffeomorphism $\M\cong\M_x$ for all $x\in V(\Sigma,\,\sigma)$.
\end{proof}
In fact, in the proof we only used the properties of the graph $\Gamma(\Sigma,\,\sigma)$, and the fact that $\Sigma\mi\sigma$ is a collection of punctured spheres. This suggests that if $\Sigma\cong\R P^2$, a similar statement holds.\\

Let $(\R P^2,\,\sigma)$ be a real projective plane admitting a log-symplectic structure with singular locus $\sigma$. Let $j$ be a complex structure on $T\R P^2(-\log(\sigma))$, and let $u:(\R P^2,\,\sigma,\,j)\longrightarrow (X,\,Z,\,J)$ be holomorphic.   
\begin{dfn}
	We say that a holomorphic map $u:(\R P^2,\,\sigma,\,j)\longrightarrow (X,\,Z,\,J)$ has $0$-self intersection if each component of $u|_{\R P^2\mi\sigma}$ has $0$-self intersection as a punctured holomorphic curve.
\end{dfn}

\begin{thm}\label{thmmodulilogproj}
	Let $(X,\,Z,\,\omega)$ be a log-symplectic $4$-manifold. Assume that there is a neighbourhood of $Z$ realizing $Z$ as a union of $S^1\times \Sigma_{g_i}$, with the trivial stable Hamiltonian structure. Let $J$ be a cylindrical complex structure with respect to such trivialization. Then the moduli space $\M$ of nice $J$-holomorphic projective planes with $0$ self-intersection is a smooth closed $2$-dimensional manifold.
\end{thm}
\begin{proof}
	The proof is almost exactly the same as \Cref{thmmodulilogspheres}. We can split a holomorphic projective plane as a union of punctured holomorphic spheres, and apply automatic transversality to make sure $J$ is generic. One then constructs a moduli space for each components, and glues along the singular locus. The only difference with the above is that the graph $\Gamma(\Sigma,\,\sigma)$ contains a circle. However, the circle does not prescribe any glueing condition between moduli spaces (moduli spaces are associated to vertices, and two vertices are glued using the edges joining them). 
\end{proof}
One can construct a moduli space of curves with one marked point using a similar gluing construction. A delicate point is that the compactification of the spaces of curves with a marked point is a manifold with boundary, and we can only construct a $C^0$ structure at the boundary points. This reflects the fact that we are constructing the moduli space of holomorphic curves by gluing punctured spheres along their asymptotic orbits. However we don't know in general how to control the infinite jet of the punctured spheres along their asymptotic orbits; hence we can only glue their images as topologically embedded surfaces. 
\begin{thm}\label{thmmodulimarkedlogcurves}
	Same assumptions as \Cref{thmmodulilogspheres} or \Cref{thmmodulilogproj}. There exist moduli spaces $\M_1$ of log-holomorphic spheres and projective planes with one marked point, which is a closed topological $4$-dimensional manifold. It comes with a canonical smooth structure on the complement of a codimension-$1$ submanifold $\M_Z$. Moreover there is a well defined evaluation map $\text{ev}:\M_1\longrightarrow X$, such that $\text{ev}\1(Z)=\M_Z$, which is continuous, and smooth on $\M_1\mi\M_Z$. 
\end{thm}
\begin{proof}
	For each $x\in V(\Sigma,\,\sigma)$, define $\overline{\M}_{1\,x}$ as the compactification of the space of holomorphic curves (in the appropriate component) with one marked point. This is a manifold whose boundary can be identified with a union of copies of $\overline{\M}_x\times S^1$, one for each edge having $x$ as a vertex. Using $r_{yx}\1 \circ r_{xy}$ to identify $\overline{\M}_x\times S^1$ with $\overline{\M}_y\times S^1$, define $\overline{\M}_1$ as the space obtained by glueing the manifolds $\overline{\M}_{1\,x}$ along their boundaries, using the identifications just explained, as dictated by the graph $\Gamma(\Sigma,\,\sigma)$, in the same fashion as in the proof of \Cref{thmmodulilogspheres}. The resulting space is a topological manifold, with the properties claimed in the statement.
\end{proof}
The following corollary is what was described in the Introduction as \Cref{C}.
\begin{coro}\label{loglogruled}
	Let $(X,\,Z,\,\omega)$ be a log-symplectic $4$-manifold satisfying the assumptions from \Cref{thmmodulilogspheres}. Assume further that $X$ is minimal, and that there exists a cylindrical complex structure $J$ (with respect to the trivialization in the assumption) and a $J$ holomorphic sphere with $0$ self-intersection, or a $J$-holomorphic projective plane with $0$ self intersection. Then there is a continuous map $f:X\longrightarrow B$, with $B$ a smooth closed orientable surface, such that
	\begin{itemize}
		\item $f|_{X\mi Z}$ is a smooth submersion;
		\item $Z$ is a union of copies of $S^1\times B$, and $f|_Z$ is the projection to $B$;
		\item the fibers of $f$ are spheres if $X$ is orientable, projective planes if $X$ is not orientable;
		\item the fibers are symplectic on $X\mi Z$, and the closures in $X$ of the components of $f\1(p)\cap X\mi Z$ are log-symplectic surfaces with boundary, of which the boundary is the singular locus.
	\end{itemize}
\end{coro}
\begin{proof}
	There is an obvious map $g:\M_1\longrightarrow \M$, which can be described as follows. Let $(u,\,z)\in\M_1\mi \M_Z$, and let $x\in V(\Sigma,\,\sigma)$ correspond to the component that contains the point $z$. Let $u_x$ be the restriction of $u$ to that component. This defines a unique element $(u_x)_x\in \M$ (via the canonical isomorphism $\M_x\cong \M$). This map extends continuously to $\M_Z$, by the way we chose the gluing maps.\\
	
	In order to prove the theorem, now, it suffices to show that the evaluation $\text{ev}:\M_1\longrightarrow X$ is a homeomorphism, which is a diffeomorphism on $\M_1\mi Z$. This follows from the same arguments as in \Cref{preliminary}. 
\end{proof}

\begin{rk}
	In \Cref{logruled} we were able to produce a moduli space of holomorphic curves starting from a symplectic submanifold. In the setting of \Cref{loglogruled} instead we need to start from a holomorphic curve. This is because we know how to construct good moduli spaces of $J$-holomorphic curves for a certain type of cylindrical complex structure. If one starts from, say, a log-symplectic submanifold, it is guaranteed that one will find a log-complex structure $J$ making it holomorphic, but $J$ won't be cylindrical in general.\\
	If $Z\cong S^1\times S^2$ one can in fact deform any log-symplectic curve to an asymptotically cylindrical one. If it was possible to construct a moduli space of such spheres with a \textit{smooth} evaluation map, then one could mimic Gromov's argument (\cite{gromov85}, see also \cite{wendlruled}) to show that all genus-$0$ log-symplectic structures on $S^2\times S^2$ are symplectomorphic to products. Unfortunately at present we do not know how to ensure smoothness, and the genus-$0$ case does not yield to stronger results.
\end{rk}

\appendix
\section{Intersection theory of punctured holomorphic curves}\label{sectionintersectionpunctured}

In this appendix we briefly introduce some aspects of the intersection theory of punctured holomorphic curves. This is developed in \cite{siefring} in the non-degenerate case, and \cite{wendlpunctured} and \cite{siefringwendl} in the Morse-Bott case.

\subsection{Intersection number}

Let $(W,\,\partial W,\,\omega)$ be a $4$-dimensional symplectic manifold, and assume $\partial W$ inherits a Morse-Bott stable Hamiltonian structure. Let $u$ be a punctured holomorphic curve, with set of punctures $\Gamma$.\\
We assume for simplicity that all Reeb orbits are simply covered.
\begin{thm}\label{thmdefpuncturedintersection}
	Given two distinct holomorphic curves $u,\,u'$, with asymptotic constraints $\mathfrak{c},\,\mathfrak{c}'$, there exists a number $i_\infty(u,\,u')$ such that the pairing
	\begin{equation}
	u\star u':=u\cdot u'+i_\infty(u,\,u')\in \Z
	\end{equation}
	(where $u\cdot u'$ is the number of intersection points, with multiplicity) depends only on the components of the moduli spaces $\M^\mathfrak{c}$, $\M^\mathfrak{c'}$ containing respectively $u$ and $u'$.
\end{thm}
The number $u\star u'$ is the \textit{intersection number} between the two punctured holomorphic curves. The dependence only on the corresponding components of the moduli spaces can be stated as the invariance of the intersection number under homotopies of punctured holomorphic curves with asymptotics $\mathfrak{c}$, $\mathfrak{c'}$. The intersection number does depend on the choice of asymptotic constraint (i.e., whether we consider a punctured to be constrained or unconstrained), so it really should be seen as a pairing between pairs of curves together with constraints.\\

A way to compute the intersection number is by relating it to the spectrum of the asymptotic operators. Let us introduce some notation. 
Let $\tau$ denote a choice of asymptotic trivialization of the hyperplane distribution for each puncture. Consider the asymptotic operator associated to a Reeb orbit $\gamma$, which in the trivialization $\tau$ can be written as $A_\gamma^\tau=-J_0\frac{d}{dt}-S(t)$. For all $|\delta|$ small enough, the asymptotic operator $A_\gamma^\tau+\delta:=-J_0\frac{d}{dt}-S(t)+\delta\mathbb{I}$ is non-degenerate. Moreover, the extremal winding numbers $\alpha_\pm(A_\gamma^\tau+\delta)$ (see \Cref{equationextremalwinding}) are independent on $\delta$ if $|\delta|$ is small enough. We define the extremal winding numbers of a Reeb orbit as
\begin{equation}
\alpha_\pm^\tau(\gamma+\delta):=\alpha_\pm(A_\gamma^\tau+\delta)
\end{equation}
In order to write a computable formula for the $\star$ pairing, we introduce the following two numbers:
\begin{itemize}
	\item $u\bullet_\tau u'$ is the intersection number of $u$ and a perturbation of $u'$ ``in the direction of $\tau$". By definition, the perturbation is supported in a neighbourhood of infinity, and is of the following type. Assume $u'$ is asymptotic to $\gamma$ at $z$. We deform $u'$ so that it is asymptotic to $\exp_{\gamma(t)}\eta(t)$, for $\eta\in\Gamma(\gamma^*\xi)$ small, with $\text{wind}(\eta)=0$. One can ensure that $u$ and $u'$ intersect in a finite number of points.
	\item Let $z,\,z'$ be punctures of the domain of $u,\,u'$ respectively, asymptotic to $\gamma,\,\gamma'$. Define
	\begin{flalign}
	\Omega^\tau_\pm(\gamma+\delta,\,\gamma'+\delta)=
	\left\{\begin{aligned}
	& 0  &\text{if }\gamma\neq\gamma'\\
	& \mp\alpha^\tau_\mp(\gamma+\delta) &\text{if } \gamma=\gamma'
	\end{aligned}\right.
	\end{flalign} 
\end{itemize}
Now, given a puncture $z$, and $\varepsilon>0$ suitably small, define
\begin{flalign}
\delta_z=
\left\{\begin{aligned}
& \varepsilon  &\text{if }z\in\Gamma_C\\
-& \varepsilon &\text{if }z\in\Gamma_U
\end{aligned}\right.
\end{flalign} 
And we define the \textit{perturbed asymptotic operator} as $\{A_z\pm \delta_z\}_{z\in\Gamma^\pm}$
\begin{thm}\label{thmintersectioneigenvalues}
	\begin{equation}
	u\star u'=u\bullet_\tau u'-\sum_{(z,\,z')\in \Gamma^\pm\times \Gamma'^\pm}^{}\Omega^\tau_\pm(\gamma_z\pm\delta_z,\,\gamma_{z'}\pm\delta_{z'})
	\end{equation}
\end{thm}
Note that this allows to define the self-intersection number $u\star u$.

\subsection{Adjunction formula}

Before to state the adjunction formula, we need to introduce some notation. Define the \textit{normal Chern number} of a punctured holomorphic curve $u$ as
\begin{equation}
c_N(u)=c_1^\tau(u)-\chi(\dot{\Sigma})+\sum_{z\in\Gamma^+}\alpha^\tau_-(\gamma_z+\delta_z)-\sum_{z\in\Gamma^-}\alpha^\tau_+(\gamma_z-\delta_z)
\end{equation}
Since $A_{\gamma_z}\pm\delta_z$ is non-degenerate, \Cref{equationparity} holds. We can write $\Gamma=\Gamma_0\sqcup\Gamma_1$, with $z\in\Gamma_i\cap\Gamma^\pm$ if and only if $p(A_{\gamma_z}\pm \delta_z)=i$ ($\Gamma_0$ and $\Gamma_1$ are said the sets of \textit{even/odd punctures}). Using \Cref{indexformulasft}, one easily checks that
\begin{equation}
2c_N(u)=\text{ind}(u)-\chi(\dot{\Sigma})-\#\Gamma_1
\end{equation}
\begin{thm}[Adjunction formula]\label{thmadjunctionpunctured}
	For a holomorphic curve with constraint $\mathfrak{c}$, with value in a $4$-manifold with cylindrical ends, there exist a homotopy invariant number $\text{sing}(u;\,\mathfrak{c})\geq 0$, such that $u$ is embedded whenever $\text{sing}(u;\,\mathfrak{c})=0$, and
	\begin{equation}
	u\star u=2\text{sing}(u;\,\mathfrak{c})+c_N(u)=2\text{sing}(u;\,\mathfrak{c})+\frac{1}{2}(\text{ind}(u)-\chi(\dot{\Sigma})-\#\Gamma_1)
	\end{equation}
\end{thm}

\subsection{Computation of intersection numbers}\label{computationintersections}

In this section we compute the intersection numbers used in the proof of \Cref{theoremmoduliboundary0intersection}. Until the end of the section, we will be considering a symplectic manifold $W$ with boundary $\partial W$, inheriting a stable Hamiltonian structure $(\alpha,\,\beta)$ such that $(\partial W,\,\alpha,\,\beta)\cong (S^1\times \Sigma_g,\,cdt,\,\text{vol}_{\Sigma_g})$. In particular, the Reeb vector field is $\frac{1}{c}\frac{\partial}{\partial t}$, and the simple Reeb orbits are $\gamma(t)=(t,\,x_0)$. The pullback bundle $\gamma^*\xi:=\gamma^*\ker\alpha\cong S^1\times T_{x_0}\Sigma_g$, hence a trivialization $\tau$ can be given by a Hermitian isomorphism $T_{x_0}\Sigma_g\cong \C$. We will always assume that this choice of trivialization is fixed, unless otherwise stated.\\
Finally, we will denote the completion of $W$ (\Cref{paragraphstablehamiltonian}) with $\widehat{W}$.\\

Let us compute the asymptotic operator and its spectrum, in order to compute intersection numbers using \Cref{thmintersectioneigenvalues}. One sees immediately from \Cref{defasymptoticoperator} that $A_\gamma=-J_0\frac{d}{dt}$. As expected, the kernel is $2$-dimensional (i.e. equal to the dimension of the space of Reeb orbits), consisting of all constant complex functions. All eigenvalues are $2\pi k$, $k\in\Z$ with eigenfunctions $C\exp(2\pi kJ_0t)$, $C\in\C$.\\

Taking $\varepsilon$ small ($\varepsilon<2\pi$), consider $A_\gamma\pm\varepsilon$. This operator has eigenvalues $\lambda^\pm_k=2\pi k\pm \varepsilon$, with eigenfunctions $f=C\exp(2\pi k t J_0)=C\exp((\lambda^\pm_k\mp\varepsilon)tJ_0)$. The corresponding extremal winding numbers are
\begin{equation}\label{computationextremalwinding}
\begin{aligned}
\alpha_+(A+\varepsilon)&=\text{min}\{\text{wind}(\lambda^+_k):\,\lambda_k>0\}=\text{min}\{\text{wind}(\exp(2\pi ktJ_0))=k:\,k\geq 0\}=0\\
\alpha_-(A+\varepsilon)&=\text{max}\{\text{wind}(\lambda^+_k):\,\lambda_k<0\}=\text{max}\{\text{wind}(\exp(2\pi ktJ_0))=k:\,k<0\}=-1\\
\alpha_+(A-\varepsilon)&=\text{min}\{\text{wind}(\lambda^-_k):\,\lambda^-_k>0\}=\text{min}\{\text{wind}(\exp(2\pi ktJ_0))=k:\,k\geq 1\}=1\\
\alpha_-(A-\varepsilon)&=\text{max}\{\text{wind}(\lambda^-_k):\,\lambda^-_k<0\}=\text{max}\{\text{wind}(\exp(2\pi ktJ_0))=k:\,k\leq 0\}=0
\end{aligned}
\end{equation}
In particular this implies
\begin{lemma}
	All punctured holomorphic curves $u:\dot{\Sigma}=\Sigma\mi\Gamma\longrightarrow W$ with simple asymptotics have only odd punctures.
\end{lemma}
\begin{proof}
	For notational simplicity, we prove this for $z\in\Gamma^+$ a positive puncture (the computation is identical for negative punctures). Recall that the parity of a Morse-Bott puncture asymptotic to an orbit $\gamma$ is defined as $p(A_z)=\alpha_+(A_z\pm \varepsilon)-\alpha_-(A_z\pm\varepsilon)$, where the sign depends on whether we see the puncture as constrained or unconstrained. Either way, \Cref{computationextremalwinding} gives $p(A_z)=1$.
\end{proof}
\begin{coro}\label{indexfromintersection}
	Let $u:\dot{\Sigma}\longrightarrow \widehat{W}$ be an embedded punctured holomorphic sphere. Then
	\begin{equation}
	\text{ind}(u)=2+2(u\star u)
	\end{equation}
\end{coro}
\begin{proof}
	The equation follows immediately from the adjunction formula \Cref{thmadjunctionpunctured} and the previous lemma.
\end{proof}
Using formula \ref{thmconleyzehnderwinding} and (\ref{computationextremalwinding}) one sees immediately that
\begin{equation}
\mu_{CZ}(-J_0\frac{d}{dt}\pm \varepsilon\mathbb{I})=\mp 1
\end{equation}
As a result we can compute the total Conley-Zehnder index of a punctured holomorphic curve.
\begin{lemma}
	Let $(W,\,\partial W,\,\omega)$ be symplectic with stable Hamiltonian boundary $(\underset{i}{\sqcup} S^1\times \Sigma_{g_i},\,c_idt,\,\text{vol}_{\Sigma_{g_i}})$. Let $\mathfrak{c}$ be a set of constraints, consisting of simple orbits. Let $\Gamma$ be a set of punctures, split $\Gamma=\Gamma^\pm_C\cup\Gamma^\pm_U$. Then 
	\begin{equation}
	\mu(\mathfrak{c})=\#\Gamma_U-\#\Gamma_C
	\end{equation}
\end{lemma}
\begin{proof}
	Recall that
	\begin{equation}
	\mu(\mathfrak{c}):=\sum_{z\in\Gamma^+}^{}\mu_{CZ}(\gamma_z+\delta_z)-\sum_{z\in\Gamma^-}\mu_{CZ}(\gamma_z-\delta_z)
	\end{equation}
	(see \Cref{deftotalconleyzehnder}), where
	\begin{flalign}
	\delta_z=
	\left\{\begin{aligned}
	& \varepsilon  &\text{if }z\in\Gamma^\pm_C\\
	-& \varepsilon &\text{if }z\in\Gamma^\pm_U
	\end{aligned}\right.
	\end{flalign} 
	Hence
	\[
	\mu(\mathfrak{c})=\sum_{z\in\Gamma^+_U}^{}\mu_{CZ}(\gamma_z-\varepsilon)+\sum_{z\in\Gamma^+_C}^{}\mu_{CZ}(\gamma_z+\varepsilon)-\sum_{z\in\Gamma^-_U}^{}\mu_{CZ}(\gamma_z+\varepsilon)-\sum_{z\in\Gamma^-_C}^{}\mu_{CZ}(\gamma_z-\varepsilon)=\]
	\[
	\sum_{z\in\Gamma^+_U}^{}1+\sum_{z\in\Gamma^+_C}^{}(-1)-\sum_{z\in\Gamma^-_U}^{}(-1)-\sum_{z\in\Gamma^-_C}^{}1=\#\Gamma_U-\#\Gamma_C
	\]
\end{proof}
Finally, we conclude by compputing the difference between self intersection numbers in the constrained and unconstrained case. Denote by $u\star_U u$ the self intersection number of a curve $u$ where all the asymptotic orbits are considered unconstrained. Analogously, $u\star_C u$ is the intersection number where all the orbits are considered constrained.
\begin{lemma}\label{comparisonconstrainedunconstrained}
	\begin{equation}
	u\star_U u=u\star_C u+\#\Gamma
	\end{equation}
\end{lemma}
\begin{proof}
	We use \Cref{thmintersectioneigenvalues}. Then \[u\star_U u-u\star_C u=-\sum_{z\in\Gamma^\pm}^{}\mp\alpha_\mp(\gamma_z\mp\varepsilon)+\sum_{z\in\Gamma^\pm}^{}\mp\alpha_\mp(\gamma_z\pm\varepsilon)=-\sum_{z\in\Gamma^\pm}^{}0+\sum_{z\in\Gamma^\pm}^{}1=\#\Gamma\]
\end{proof}

\printbibliography

\end{document}